\documentclass[11pt]{article}
\usepackage[a4paper,hmargin={2.5cm,2.5cm},vmargin={2.5cm,2.5cm}]{geometry}

\usepackage[utf8]{inputenc}
\usepackage{lipsum}
\usepackage[T1]{fontenc}
\usepackage{lmodern}
\usepackage{mathrsfs}
\usepackage[english]{babel}
\usepackage{csquotes}
\usepackage{hyperref}
\hypersetup{hidelinks}
\usepackage[shortlabels]{enumitem}
\usepackage{xcolor}
\usepackage[leqno]{amsmath}
\usepackage{amssymb,amsthm}

\title{Existence results for variational quasilinear elliptic systems involving the vectorial \texorpdfstring{$p$}{p}-Laplacian}
\date{\today}
\author{Annamaria Canino and Simone Mauro\thanks{Corresponding author: \href{mailto:simone.mauro@unical.it}{\texttt{simone.mauro@unical.it}}.}\protect\\[3pt]
{\protect\small Dipartimento di Matematica e Informatica, Università della Calabria,}\protect\\
{\protect\small Ponte Pietro Bucci cubo 31B, 87036 Arcavacata di Rende, Cosenza, Italy}\protect\\[3pt]
{\protect\small\protect\texttt{annamaria.canino@unical.it} \protect\quad \protect\texttt{simone.mauro@unical.it}}}

\numberwithin{equation}{section}

\usepackage{tikz-cd}

\theoremstyle{plain}
\newtheorem{theorem}{Theorem}[section]
\newtheorem{lemma}[theorem]{Lemma}
\newtheorem{proposition}[theorem]{Proposition}
\newtheorem{corollary}[theorem]{Corollary}

\theoremstyle{plain}
\newtheorem{definition}[theorem]{Definition}

\theoremstyle{plain}
\newtheorem{remark}[theorem]{Remark}

\RequirePackage{dsfont}
\newcommand{\N}{\field{N}}

\newcommand{\R}{\field{R}}

\newcommand{\field}[1]{\mathbb{#1}}
\newcommand{\wto}{\rightharpoonup}

\tikzset{%
    symbol/.style={%
        draw=none,
        every to/.append style={%
            edge node={node [sloped, allow upside down, auto=false]{$#1$}}}
    }
}
\usepackage{enumitem}

\tikzset{shorten <>/.style={shorten >=#1,shorten <=#1}}

\usepackage[swapnames,suftesi]{frontespizio}

\begin{document}

\maketitle

\begin{abstract}
    We prove  existence and regularity results for weak solutions of the following elliptic system:
\[
\begin{cases}
-\textbf{div}(|D\boldsymbol{u}|^{p-2}D\boldsymbol{u})=\boldsymbol{f}(x,\boldsymbol{u}) & \text{in } \Omega, \\  
\boldsymbol{u}=0 & \text{on } \partial\Omega,  
\end{cases}
\]
where $\boldsymbol{u}=(u^1,\dots,u^m)$, $p>1$, and $\Omega\subset\mathbb{R}^N$ is a bounded domain.  
We also consider the special case
 \[\boldsymbol{f}(x,\boldsymbol{u})=\lambda|\boldsymbol{u}|^{p-2}\boldsymbol{u}+|\boldsymbol{u}|^{q-2}\boldsymbol{u},\] 
 and we prove a classification result. 
 In particular, we show that any least energy solution is of the form $(c^1\omega,\dots,c^m\omega)$, where $\boldsymbol{c}=(c^1,\dots,c^m)\in S^{m-1}$ (the $(m-1)$-sphere in $\mathbb R^m$) and $\omega$ is a positive solution of the corresponding scalar equation.
\end{abstract}

\section{Introduction}
Let $\Omega\subset\R^N$ be a bounded domain with $N\ge2$, let $m>1$ be an integer, and let $p>1$. We consider the functional $\mathcal J:W_0^{1,p}(\Omega; \R^m)\to\R$,
\[\mathcal J(\boldsymbol{u})=\frac1p\int_\Omega |D\boldsymbol u|^p-\int_\Omega F(x,\boldsymbol u),\]
where $\boldsymbol u=(u^1,\dots,u^m)$ and $F:\Omega\times\R^m\to\R$ is a $C^1$-Carathéodory function:
\begin{itemize}
    \item $F(\cdot,\boldsymbol{s})$ is measurable for every $\boldsymbol{s}\in \R^m$,
    \item $F(x,\cdot)$ is $C^1$ for a.e. $x\in\Omega$.
\end{itemize}
We also suppose that
\[
    \tag{$f.1$}\label{g.1}
    |\nabla_{\boldsymbol s} F(x,\boldsymbol s)|\le a(x)+b|\boldsymbol s|^{q-1},\quad F(x,\boldsymbol{0})=0,
\]
with $a\in L^{r}(\Omega; \R)$, $b\ge0$ and 
\[
1<q<p^*:=
\begin{cases}
\frac{Np}{N-p}&\text{if $p<N$}\\
\infty&\text{if $p\ge N$}.
\end{cases}
\]
{
The exponent $r$ is assumed to satisfy
\begin{equation}\label{r assumption}
\begin{cases}
r\ge (p^*)'=\dfrac{Np}{Np-N+p},&\text{if $p<N$,}\\
r>1,&\text{if $p=N$,}\\
r\ge1,&\text{if $p>N$.}
\end{cases}
\end{equation}
}

{
{In this work we focus on the case $p<N$. The cases $p=N$ and $p>N$ are treated as follows.

Case $p=N$: every occurrence of the critical exponent $p^*$ is replaced by a finite exponent $\kappa$ chosen arbitrarily large, by virtue of the embedding \(W_0^{1,N}(\Omega;\R^m)\hookrightarrow L^\kappa(\Omega;\R^m)\).

Case $p>N$: one relies directly on the embedding \(W_0^{1,p}(\Omega;\R^m)\hookrightarrow L^\infty(\Omega;\R^m)\).

The necessary technical adjustments for $p=N$ and $p>N$ are explicitly highlighted in the relevant proofs below.}
}

We point out that $D\boldsymbol u$ is the Jacobian matrix
\[
D\boldsymbol{u}=
\begin{bmatrix}
    \nabla u^1\\
    \vdots\\
    \nabla u^m
\end{bmatrix}
=
\begin{bmatrix}
    \frac{\partial u^1}{\partial x_1}&\dots&\frac{\partial u^1}{\partial x_N}\\
    \vdots &&\vdots\\
    \frac{\partial u^m}{\partial x_1}&\dots&\frac{\partial u^m}{\partial x_N}
\end{bmatrix},
\qquad 
\nabla u^i=
\begin{bmatrix}
\frac{\partial u^i}{\partial x_1}\\
\vdots\\
\frac{\partial u^i}{\partial x_N}
\end{bmatrix}, \quad\text{$i=1,\dots,m$,}
\]
and we use the norm
\[|D\boldsymbol u|^2:=|\nabla u^1|^2+\dots+\left|\nabla u^m\right|^2.\]
We investigate the existence of critical points of $\mathcal J$, which correspond to weak solutions of the elliptic system:
\[
\tag{$\mathcal P$}\label{P}
\begin{cases}
   -\mathbf\Delta_p\boldsymbol u:= -\textbf{div}(|D\boldsymbol u|^{p-2}D\boldsymbol u)=\boldsymbol f(x,\boldsymbol u)&\text{in $\Omega$,}\\
    \boldsymbol u=0&\text{on $\partial\Omega$},
\end{cases}
\]
where $\boldsymbol f(x,\boldsymbol s):=\nabla_{\boldsymbol s} F(x,\boldsymbol s)$, 
\begin{align*}
&\textbf{div}(|D\boldsymbol{u}|^{p-2}D\boldsymbol{u}):=
\begin{bmatrix}
    \text{div}(|D\boldsymbol{u}|^{p-2}\nabla u^1)\\
    \vdots\\
    \text{div}(|D\boldsymbol{u}|^{p-2}\nabla u^m)
\end{bmatrix},
&\boldsymbol{f}(x,\boldsymbol{u})=
\begin{bmatrix}
    f^1(x,\boldsymbol{u})\\
    \vdots\\
    f^m(x,\boldsymbol{u})
\end{bmatrix}.
\end{align*}
We also note that \eqref{g.1} implies
\[|f^j(x,\boldsymbol{s})|\le a(x)+b|\boldsymbol{s}|^{q-1},\ \ \text{for every $j=1,\dots,m$.}\]
Furthermore, the vectorial $p$-Laplacian equation \eqref{P} can be written componentwise as
\[
\begin{cases}
    -\text{div}(|D\boldsymbol{u}|^{p-2}\nabla u^i)=f^i(x,u^1,\dots,u^m)&\text{in $\Omega$, \quad $i=1,\dots,m$,}\\
    u^1=\dots=u^m=0&\text{on $\partial\Omega$.}
\end{cases}
\]
Hence, \eqref{P} is a coupled system in which the coupling appears both in the elliptic operator on the left-hand side and in the nonlinearity on the right-hand side.

Nonlinear systems involving the vectorial $p$-Laplacian naturally arise in the modeling of diffusion processes in which the constitutive relation between fluxes and gradients is nonlinear. A paradigmatic example is provided by the equations governing incompressible viscous flows, which can be derived starting from the balance of linear momentum expressed in terms of the Cauchy stress tensor {\cite[Sections 3.3.1--3.3.2]{irgens2014rheology} and \cite[Section 1.1]{ruzicka2000electrorheological}}.

Let $\boldsymbol u : \Omega \times (0,T) \to \mathbb{R}^m$, $m\le 3$, denote the velocity field, and let $\boldsymbol \sigma$ be the stress tensor. For an incompressible fluid with constant density, the balance of momentum reads
\begin{equation*}
\partial_t \boldsymbol u + (\boldsymbol u \cdot \nabla)\boldsymbol u
= \textbf{div}(\boldsymbol \sigma) + \boldsymbol f
\quad \text{in } \Omega \times (0,T),
\end{equation*}
together with the incompressibility condition
\begin{equation*}
\textbf{div}(\boldsymbol u) = 0.
\end{equation*}

In the Newtonian case, the stress tensor is assumed to depend linearly on the velocity gradient and is given by
\begin{equation*}
\boldsymbol \sigma = -\mathfrak p I + \nu\, D\boldsymbol u,
\end{equation*}
where $\mathfrak p$ denotes the pressure and $\nu>0$ the kinematic viscosity. Inserting this relation into the momentum balance yields the classical Navier--Stokes equations {\cite[Section 3.4]{irgens2014rheology}; see also \cite[Sections 1.1 and 1.3]{ruzicka2000electrorheological}}
\begin{equation*}
\partial_t \boldsymbol u - \nu\,\boldsymbol\Delta \boldsymbol u
+ (\boldsymbol u \cdot \nabla)\boldsymbol u
+ \nabla \mathfrak p
= \boldsymbol f,
\qquad
\textbf{div}(\boldsymbol u) = 0.
\end{equation*}

Neglecting inertial effects leads to the Stokes system, obtained by dropping the nonlinear convective term $(\boldsymbol u \cdot \nabla)\boldsymbol u$.

In several applications, however, the linear dependence on the gradient is no longer adequate and nonlinear diffusion effects must be taken into account. A prototypical model is provided by operators of $p$-Laplacian type. {Such operators arise naturally in the mathematical description of non-Newtonian and power-law fluids; see \cite[p. 296 and Section 2.1]{barrenechea2023powerlaw}, \cite[Sections 1.4 and 6.1]{irgens2014rheology}, and \cite[Section 1.4]{ruzicka2000electrorheological}.} To focus on the main analytical features of this nonlinear behavior, we consider diffusion terms depending on the full gradient, leading to the vectorial $p$-Laplacian
\[
-\textbf{div}\bigl(|D\boldsymbol u|^{p-2}D\boldsymbol u\bigr),
\qquad p>1,
\]
This operator can be regarded as a natural nonlinear generalization of the classical Laplacian and reduces to the linear case when $p=2$.


{In the scalar case $m=1$, existence and multiplicity for Dirichlet problems driven by the $p$-Laplacian are classical topics. A variational and topological treatment covering, among other cases, existence under $p$-sublinear growth and nontrivial or multiple solutions under $p$-superlinear assumptions is given in \cite[Theorems 11, 13, 18 and 20]{dinca2001variational}. The scalar multiplicity result \cite[Corollary 1.1]{candela2009infinitely} subsequently removed the nonresonance condition at the origin imposed in \cite[Theorem 20]{dinca2001variational}. Scalar perturbations of the $p$-homogeneous principal part, including interactions with the variational spectrum, were studied in \cite{costa1995perturbations}. Scalar boundedness and regularity results under subcritical growth can be found in \cite{pardo2024priori,pucci2008regularity}, while scalar Lane--Emden ground states are discussed, for example, in \cite{grumiau2008asymptotics}. The results below are compared in detail with these scalar predecessors after each statement.}

We denote by $\lambda_1$ the first eigenvalue of the vectorial $p$-Laplacian under homogeneous Dirichlet boundary conditions, namely
\[
\lambda_1:=\inf\left\{\int_\Omega|D\boldsymbol u|^p\ :\ \boldsymbol u\in S_p\right\},
\qquad
S_p:=\left\{\boldsymbol u\in W_0^{1,p}(\Omega;\R^m)\ :\ \int_\Omega|\boldsymbol u|^p=1\right\}.
\]
{The equivalence between this variational characterization and the spectral definition of the first eigenvalue is proved in Proposition \ref{minimum implies eigenfunction}; see also \cite[Definition 1.1 and Theorem 1.2]{hynd2023uniqueness}.}

We prove existence and regularity for weak solutions. We first establish the following existence results:
\begin{theorem}\label{thm 1}
    Let $1<q<p$ and assume that \eqref{g.1} holds. Then $\mathcal J$ has a global minimizer $\boldsymbol{u}$, which is a weak solution of \eqref{P}. Furthermore, if
    \[
    \tag{$f.2$}\label{g.3}
    \liminf_{\boldsymbol{s}\to\boldsymbol{0}}\frac{pF(x,\boldsymbol{s})}{|\boldsymbol{s}|^{p}}>\lambda_1\quad \text{uniformly for a.e. $x\in\Omega$},\quad \boldsymbol{f}(x,\boldsymbol{0})=\boldsymbol{0},\]
  then the minimizer is nontrivial, i.e. $\boldsymbol u\not\equiv\boldsymbol 0$.
\end{theorem}
{For $m=1$ and $1<q<p$, scalar existence results under comparable sub-$p$ growth assumptions were established in \cite[Theorem 11]{dinca2001variational} via a fixed-point argument and in \cite[Theorem 13 and Remark 8(i)]{dinca2001variational}  using a coercive variational approach. Theorem \ref{thm 1} yields the corresponding result for the vectorial functional, identifies the solution as a global minimizer, and provides the explicit condition \eqref{g.3} ensuring that this minimizer is nontrivial.}
\begin{remark}
 Theorem \ref{thm 1} applies to the case $\boldsymbol f(x,\boldsymbol u)=\boldsymbol{f}(x)\in L^r(\Omega; \R^m)$ {with $r$ satisfying \eqref{r assumption}}, and to the $p$-sublinear case $\boldsymbol{f}(x,\boldsymbol{u})=|\boldsymbol{u}|^{q-2}\boldsymbol{u}$ with $q\in(1,p)$.
 \end{remark}
\begin{theorem}\label{thm 2}
    Let $p<q<p^*$. Assume that $\boldsymbol{f}(x,-\boldsymbol{s})=-\boldsymbol{f}(x,\boldsymbol{s})$, \eqref{g.1} holds and there exist $R>0$ and $\mu>p$ such that, for a.e. $x\in\Omega$ and for every $\boldsymbol s$ with $|\boldsymbol s|\ge R$,
\[
\tag{$f.3$}\label{g.2}
 0<\mu F(x,\boldsymbol s)\le \boldsymbol s\cdot \boldsymbol f(x,\boldsymbol s),\ \boldsymbol{f}(x,\boldsymbol{0})=\boldsymbol{0}.
\]  
 Then problem \eqref{P} has infinitely many weak solutions.  
\end{theorem}
{In the scalar case, the result in \cite[Corollary 1.1]{candela2009infinitely} provides infinitely many solutions under assumptions corresponding to oddness, subcritical growth and the Ambrosetti--Rabinowitz condition. For $m=1$, Theorem \ref{thm 2} recovers this result, whereas for $m>1$ it extends it to the vectorial $p$-Laplacian and to genuinely vector-valued nonlinearities, whose components are coupled through both $|D\boldsymbol u|$ and $\boldsymbol f(x,\boldsymbol u)$. Furthermore, \eqref{g.1} allows $a\in L^r(\Omega)$, whereas \cite[Corollary 1.1]{candela2009infinitely} requires a uniform constant bound.}

We also prove the corresponding result for a $p$-linear perturbation:
\begin{theorem}\label{thm lambda}
  Assume that $\boldsymbol{f}(x,-\boldsymbol{s})=-\boldsymbol{f}(x,\boldsymbol{s})$, \eqref{g.1} and \eqref{g.2} hold.  The problem
    \[
    \begin{cases}
    -\boldsymbol{\Delta}_p\boldsymbol{u}=\lambda|\boldsymbol{u}|^{p-2}\boldsymbol{u}+\boldsymbol{f}(x,\boldsymbol{u})&\text{in $\Omega$,}\\
\boldsymbol{u}=0&\text{on $\partial\Omega$}
    \end{cases}
    \]
    has infinitely many weak solutions for every $\lambda\in\R$.
\end{theorem}
{In the scalar setting, \cite[Corollary 1.1]{candela2009infinitely} establishes a multiplicity result, whereas \cite[Theorems 1.1--1.2]{costa1995perturbations} prove the existence of a solution for more general perturbations of the scalar $p$-Laplacian. Theorem \ref{thm lambda} extends this multiplicity framework to the vectorial $p$-Laplacian, ensuring infinitely many solutions for every $\lambda\in\R$.}
While
the variational arguments leading to Theorems \ref{thm 1}, \ref{thm 2}, and \ref{thm lambda} are classical in spirit,  we present full details for  the vectorial case, as the operator involves the full gradient $D\boldsymbol u$ and the resulting framework is essential for the subsequent regularity and least-energy analysis.

We note that if $\boldsymbol u \in L^\infty(\Omega;\mathbb{R}^m)$, then 
$\boldsymbol u \in C^{1,\beta}(\overline{\Omega};\mathbb{R}^m)$ by 
\cite{benedetto1989boundary}, provided that $\partial\Omega \in C^{1,\alpha}$. 
The regularity theory for quasilinear elliptic systems has been extensively investigated in the case of source terms independent of the unknown, namely when 
$\boldsymbol f(x,\boldsymbol u) = \boldsymbol f(x)$; we refer to 
\cite{balci2022pointwise,benedetto1989boundary,CianchiArma2014,cianchi2019optimal,KuusiMingionePotential,montoro2025regularity,SchmidtMinimizer14,
sciunzi2025global}
and the references therein.

In the presence of a nonlinear reaction term $\boldsymbol f(x,\boldsymbol u)$, deriving a priori $L^\infty$ bounds in the vectorial setting is more delicate than in the scalar case. Although Moser iteration can still be applied, its implementation is more delicate because additional tensorial terms arise when testing with $|\boldsymbol u|^{\gamma p}\boldsymbol u$.

For $p \ge 2$, we overcome these difficulties by combining Moser iteration with Stampacchia-type truncation arguments. However, this approach does not seem to extend directly to the range $1<p<2$.

In the range $p\ge2$, this strategy allows us to establish a regularity result for quasilinear elliptic systems with nonlinear right-hand side $\boldsymbol f(x,\boldsymbol u)$, as stated in the following theorem.

\begin{theorem}\label{regularity cap1}
Let $p\ge2$ and let $\boldsymbol u \in W_0^{1,p}(\Omega, \mathbb{R}^m)$ be a weak solution of
\begin{equation*}
\begin{cases}
-\boldsymbol{\Delta}_p\boldsymbol u = \boldsymbol{f}(x, \boldsymbol u) & \text{in } \Omega, \\
\boldsymbol u = 0 & \text{on } \partial \Omega,
\end{cases}
\end{equation*}
and assume that there exist $a\in L^\infty(\Omega;\mathbb R)$, $b\ge0$ and $q\in(p,p^*)$ such that
\begin{equation}\label{subcritical growth theorem 1.5}
|\boldsymbol f(x,\boldsymbol s)|\le a(x)+b|\boldsymbol{s}|^{q-1}.
\end{equation}
Then  $\boldsymbol u \in L^\infty(\Omega, \mathbb{R}^m)$. Furthermore, if $\partial\Omega \in C^{1,\alpha}$ with $\alpha\in(0,1)$, then $\boldsymbol u \in C^{1,\beta}(\overline{\Omega}, \mathbb{R}^m)$ for some $\beta \in (0,1)$. 
\end{theorem}
{When $m=1$, boundedness of weak solutions for scalar $p$-Laplacian equations involving subcritical Carath\'{e}odory nonlinearities was established in \cite{pardo2024priori,pucci2008regularity}. By contrast, Theorem \ref{regularity cap1} addresses  the coupled vectorial operator and a genuinely vector-valued reaction term. The restriction $p\ge2$ reflects the present vectorial Moser--Stampacchia argument.}
The regularity result obtained here relies on the approach developed in \cite{carmona2013regularity, Vannella23}, suitably adapted to the vectorial $p$-Laplacian. 
\begin{remark}\label{oss regolarità}
  If $\boldsymbol f(x,\boldsymbol s)$ satisfies the subcritical growth \eqref{g.1} with $q\in(1,p)$, then, for every $\sigma$ such that $p-q<\sigma<p^*-q$, Young's inequality gives
    \begin{align*}
        |\boldsymbol f(x,\boldsymbol s)|
        &\le a(x)+b|\boldsymbol s|^{q-1}\\
        &\le a(x)+b\frac{\sigma}{q+\sigma-1}
        +b\frac{q-1}{q+\sigma-1}|\boldsymbol s|^{q+\sigma-1}.
    \end{align*}
    Since $p<q+\sigma<p^*$, the assumption \eqref{subcritical growth theorem 1.5} holds.
\end{remark}

Next, we focus on least energy solutions.
Let $X$ be the set of all nontrivial weak solutions:
\[X:=\left\{\boldsymbol{u}\in W_0^{1,p}(\Omega; \R^m)\setminus\{\boldsymbol 0\}\ :\ \mathcal J'(\boldsymbol{u})=0\right\},\]
and define $c:=\inf_X\mathcal J$ as the \emph{least energy level} of \eqref{P}. The set $X$ is non-empty under the assumptions of Theorems \ref{thm 1}, \ref{thm 2}, and \ref{thm lambda}.
 A \emph{least energy solution} is a weak solution $\boldsymbol u\in X$ such that $\mathcal J(\boldsymbol u)=c$.

We prove an existence and classification result for least energy solutions of a particular variational elliptic system:
\begin{theorem}\label{thm 3}
Let $q\in(1,p^*)\setminus\{p\}$ and $\lambda<\lambda_1$, where $\lambda_1>0$ is the first eigenvalue of $-\boldsymbol{\Delta}_p$ with Dirichlet boundary conditions. There exists a least energy solution $\boldsymbol{u}$ for 
    \[\tag{$\mathcal P_{p,q}$}\label{lane emeden system}
    \begin{cases}
        -\boldsymbol{\Delta}_p\boldsymbol{u}=\lambda|\boldsymbol{u}|^{p-2}\boldsymbol{u}+|\boldsymbol{u}|^{q-2}\boldsymbol u&\text{in $\Omega$,}\\
        \boldsymbol{u}=0&\text{on $\partial\Omega$}.
    \end{cases}
    \]
    Furthermore, $\boldsymbol{u}=(c^1\omega,\dots,c^m\omega)$ with $\boldsymbol{c}=(c^1,\dots,c^m)\in S^{m-1}$ and $\omega\in W_0^{1,p}(\Omega; \R)$ is a solution of
    \[
    \begin{cases}
        -\Delta_p\omega=\lambda\omega^{p-1}+\omega^{q-1}&\text{in $\Omega$,}\\
        \omega>0&\text{in $\Omega$,}\\
        \omega=0&\text{on $\partial\Omega$}.
    \end{cases}
    \]
\end{theorem}

{For the scalar Lane--Emden equation with $p<q<p^*$, the existence and variational characterization of ground states are classical; see, for example, \cite{grumiau2008asymptotics}, which studies ground states and least-energy nodal solutions  together with their asymptotic relation to scalar $p$-Laplacian eigenfunctions. For $m=1$, the existence part of Theorem \ref{thm 3} follows from the standard scalar minimization of the associated quotient when $\lambda<\lambda_1$. The main novelty of Theorem \ref{thm 3} lies in the vectorial classification: every least-energy solution has a fixed direction in $\R^m$ and is generated by a positive solution of the associated scalar problem.}

Our classification result is inspired by \cite{hynd2023uniqueness}, which addresses an eigenvalue problem, and by \cite{correia2016semitrivial}, which deals with a cooperative system involving the Laplacian.

\begin{remark}
We collect a few additional observations related to Theorem \ref{thm 3}:
\begin{itemize}
    \item When $q=p$, the problem reduces to the eigenvalue case, which has been investigated in detail in \cite{hynd2023uniqueness}.
    \item If $q\in(p,p^*)$, Theorem \ref{thm 2} ensures the existence of infinitely many solutions.
    \item Concerning the regularity of weak solutions, we recall that any least energy solution $\boldsymbol{u}$ of \eqref{lane emeden system} is bounded by the regularity theory for the scalar $p$-Laplacian (see \cite{pucci2008regularity} and \cite[Lemma 2.2]{pardo2024priori}). Moreover, if $\partial\Omega\in C^{1,\alpha}$, one deduces that $\boldsymbol{u}\in C^{1,\beta}(\overline\Omega; \R^m)$ thanks to \cite{benedetto1989boundary}.
    \item A least energy solution is not necessarily positive, even though $\omega$ is positive, since some of the coefficients $c^i$ may be non-positive.
    \item Finally, we observe that in general $\boldsymbol{c}\neq0$, but some of the components $c^i$ may vanish for certain $i\in\{1,\dots,m\}$. 
\end{itemize}
\end{remark}
\subsection{Notation}
Let $X$ be a Banach space, and denote by $\|\cdot\|_X$ its norm. 
If $X=L^p(\Omega;\mathbb{R}^m)$, with $\Omega \subset \mathbb{R}^N$, we set
\[
|\boldsymbol u|_p := \|\boldsymbol u\|_{L^p(\Omega;\mathbb{R}^m)}, \quad \forall\, u \in L^p(\Omega;\mathbb{R}^m).
\]
If instead $X=W_0^{1,p}(\Omega;\mathbb{R}^m)$, we simply write
\[
\|\boldsymbol u\| := \|\boldsymbol u\|_{W_0^{1,p}(\Omega;\mathbb{R}^m)}, \quad \forall\, u \in W_0^{1,p}(\Omega;\mathbb{R}^m).
\]
We denote by $W^{-1,p'}(\Omega;\mathbb{R}^m)$ the dual space 
$(W_0^{1,p}(\Omega;\mathbb{R}^m))'$, where $\tfrac{1}{p} + \tfrac{1}{p'} = 1$. 
When no confusion arises, we simply write $W_0^{1,p}(\Omega)$ and $L^p(\Omega)$ 
instead of $W_0^{1,p}(\Omega;\mathbb{R}^m)$ and $L^p(\Omega;\mathbb{R}^m)$. 

We also define 
\(
S^{m-1}:=\left\{\boldsymbol z\in\R^m\ :\ |\boldsymbol z|=1\right\}.
\)
\section{Preliminaries}
We recall some basic definitions and results of critical point theory of differentiable functionals.
\begin{definition}
    Let $X$ be a Banach space and let $J:X\to\R$ be a $C^1$-functional. Let $\{u_n\}\subset X$ be a sequence and let $c\in\R$. We say that $\{u_n\}$ is a Palais-Smale sequence at level $c$ if 
    \begin{itemize}
        \item $J(u_n)\to c$ in $\R$,
        \item $J'(u_n)\to0$ in $X'$, where $X'$ denotes the dual space of $X$.
    \end{itemize}
\end{definition}
\begin{definition}
    Let $X$ be a Banach space and let $J:X\to\R$ be a $C^1$-functional. We say that $J$ satisfies the Palais-Smale condition (at level $c$) if any Palais-Smale sequence (at level $c$) admits a convergent subsequence in $X$.
\end{definition}

\begin{theorem}[\cite{ambrosetti1973dual}, \cite{rabinowitz1}]\label{MP}
    Let $X$ be a Banach space and let $J:X\to\R$ be an even $C^1$-functional. Assume that  $J(0)=0$, and that there exist  $\alpha,\rho>0$ and a subspace $W\subset X$ of finite codimension such that
    \begin{itemize}
        \item[$(i)$] $J(u)\ge\alpha>0$ for every $u\in\partial B_\rho\cap W$,
        \item[$(ii)$] for every subspace $V\subset X$ of finite dimension, there exists $R>0$ such that  $J(v)\le0$ in $B_R^c\cap V$,
        \item[$(iii)$] the functional $J$ satisfies the Palais-Smale condition.
    \end{itemize}
    Then there exists a sequence of critical points $\{u_n\}$ such that $J(u_n)\to+\infty$. 
\end{theorem}
We endow $W_0^{1,p}(\Omega; \R^m)$ with the norm
\[
\|\boldsymbol{u}\|:=\left(\int_\Omega |D\boldsymbol u|^p\right)^{\frac1p}=\left(\int_\Omega\left(|\nabla u^1|^2+\dots+|\nabla u^m|^2\right)^{\frac p2}\right)^{\frac1p}.
\]
Let $\mathcal J:W_0^{1,p}(\Omega; \R^m)\to\R$ be the energy functional of \eqref{P}. We start proving the differentiability of $\mathcal J$ in this section. 
\begin{proposition}\label{frechet diff}
Assume that \eqref{g.1} holds.
   We have that $\mathcal J$ is a $C^1$-functional and 
    \[\langle\mathcal J'(\boldsymbol u),\boldsymbol{\varphi}\rangle=\int_\Omega |D\boldsymbol u|^{p-2}D\boldsymbol u\cdot D\boldsymbol\varphi-\int_\Omega \boldsymbol f(x,\boldsymbol u)\cdot\boldsymbol\varphi,\ \ \forall\  \boldsymbol u,\boldsymbol\varphi\in W_0^{1,p}(\Omega; \R^m),\]
    where the scalar product $D\boldsymbol u\cdot D\boldsymbol \varphi$ is defined as follows:
    \[D\boldsymbol u\cdot D\boldsymbol\varphi=\sum_{i=1}^m\nabla u^i\cdot\nabla\varphi^i=\sum_{j=1}^N\sum_{i=1}^m\frac{\partial u^i}{\partial x_j}\cdot\frac{\partial\varphi^i}{\partial x_j}.\]
\end{proposition}
\begin{proof}
 The proof is standard, but we provide it for the sake of clarity.\\
 \textbf{Step 1.} 
 The functional $J_1(\boldsymbol u)=\frac1p\int_\Omega|D\boldsymbol u|^p$ is of class $C^1$.
 
For every $\boldsymbol\varphi\in W_0^{1,p}(\Omega; \R^m)$, we have that
 \begin{align*}
     \lim_{t^i\to0}\frac{|D(\boldsymbol u+t^i\varphi^i\boldsymbol e_i)|^p-|D\boldsymbol u|^p}{t^i}=p|D\boldsymbol u|^{p-2}\nabla u^i\cdot\nabla\varphi^i,
 \end{align*}
 for a.e. $x\in\Omega$ and for $i=1,\dots,m$, where $\{\boldsymbol e_1,\dots,\boldsymbol e_m\}$ is the canonical basis of $\R^m$.
 
  By the mean value theorem, there exists $\theta^i\in\R$ such that $|\theta^i|\le |t^i|$ and 
\begin{align*}
     \left|\frac{|D\boldsymbol u+t^iD(\varphi^i\boldsymbol e_i)|^p-|D\boldsymbol u|^p}{t^i}\right|&= p\left||D\boldsymbol u+\theta^iD(\varphi^i\boldsymbol e_i)|^{p-2}\left(\nabla u^i+\theta^i\nabla \varphi^i\right)\cdot\nabla\varphi^i\right|.
 \end{align*}
 By the Cauchy--Schwarz inequality,
 \begin{align*}
     \left(\nabla u^i+\theta^i\nabla\varphi^i\right)\cdot\nabla\varphi^i&\le|\nabla u^i+\theta^i\nabla\varphi^i|\cdot|\nabla\varphi^i|\\
    &\le |D(\boldsymbol u+\theta^i\varphi^i\boldsymbol e_i)|\cdot|\nabla\varphi^i|.
 \end{align*}
 Thus,
 \[     \left|\frac{|D\boldsymbol u+t^iD(\varphi^i\boldsymbol e_i)|^p-|D\boldsymbol u|^p}{t^i}\right|\le p|D\boldsymbol u+\theta^i D(\varphi^i\boldsymbol e_i)|^{p-1}\cdot|\nabla\varphi^i|.\]
 We recall that
 \[|a+b|^{\gamma}\le C_{\gamma}(|a|^{\gamma}+|b|^\gamma),\qquad \forall\ a,b\in\R,\ \gamma>0.\]
 Therefore,
 \[ \left|\frac{|D\boldsymbol u+t^iD(\varphi^i\boldsymbol e_i)|^p-|D\boldsymbol u|^p}{t^i}\right|\le C\left( |D\boldsymbol u|^{p-1}|\nabla\varphi^i|+|\nabla\varphi^i|^p\right).\]
 By the dominated convergence theorem, the G\^ateaux derivative $J_1'(\boldsymbol u):W_0^{1,p}(\Omega; \R^m)\to\R$ exists and
\begin{align*}
    \langle J_1'(\boldsymbol u), \varphi^i\boldsymbol e_i\rangle=\int_\Omega |D\boldsymbol u|^{p-2}\nabla u^i\cdot \nabla\varphi^i.
\end{align*}
Thus,
\[\langle J_1'(\boldsymbol u),\boldsymbol{\varphi}\rangle=\sum_{i=1}^m\langle J_1'(\boldsymbol u),\varphi^i\boldsymbol e_i\rangle=\int_\Omega |D\boldsymbol u|^{p-2}D\boldsymbol u\cdot D\boldsymbol\varphi.\]
We claim that $J_1'$ is continuous. Indeed, let $\{\boldsymbol u_n\}\subset W_0^{1,p}(\Omega; \R^m)$ be such that
$\boldsymbol u_n\to\boldsymbol u$ in $W_0^{1,p}(\Omega; \R^m)$. By H\"older's inequality,
\[\left|\int_\Omega\left(|D\boldsymbol u_n|^{p-2}D\boldsymbol u_n -|D\boldsymbol u|^{p-2}D\boldsymbol u\right) D\boldsymbol\varphi\right|\le\left||D\boldsymbol u_n|^{p-2}D\boldsymbol u_n-|D\boldsymbol u|^{p-2}D\boldsymbol u\right|_{\frac{p}{p-1}}\cdot|D\boldsymbol\varphi|_p.\]
Since $|D\boldsymbol u_n|\to |D\boldsymbol u|$ in $L^{p}(\Omega)$, \[|D\boldsymbol u_n|^{p-2}D\boldsymbol u_n\to |D\boldsymbol u|^{p-2}D\boldsymbol u\quad \text{in $L^{\frac{p}{p-1}}(\Omega)$},\]
and the continuity follows. Hence $J_1$ is Fréchet differentiable and the Fréchet derivative coincides with the G\^ateaux derivative.  \\
\textbf{Step 2.} The functional $\mathcal J$ is of class $C^1$. 

Let $J_2(\boldsymbol u)=\mathcal J(\boldsymbol u)-J_1(\boldsymbol u)$. By the mean value theorem and \eqref{g.1}, we obtain
\begin{align*}
    \left|\frac{F(x,\boldsymbol u+t^i\varphi_i\boldsymbol e_i)-F(x,\boldsymbol{u})}{t^i}\right|&=|f^i(x,\boldsymbol{u}+\theta^i\varphi_i\boldsymbol e_i)\varphi^i|\\
    &\le a(x)|\varphi^i|+b|\boldsymbol{u}+\theta^i\varphi_i\boldsymbol e_i|^{q-1}|\varphi^i|\in L^1(\Omega),
\end{align*}
for every $\boldsymbol{\varphi}\in W_0^{1,p}(\Omega; \R^m)$ and for every $|\theta^i|\le |t^i|$. Hence, as in Step 1, the G\^ateaux derivative $J_2'(\boldsymbol u):W_0^{1,p}(\Omega; \R^m)\to\R$ exists and
\[\langle J_2'(\boldsymbol u),\boldsymbol\varphi\rangle=-\int_\Omega\boldsymbol f(x,\boldsymbol u)\cdot\boldsymbol\varphi.\]
It remains to prove that $J_2':W_0^{1,p}(\Omega; \R^m)\to W^{-1,p'}(\Omega;\R^m)$ is continuous.

{
{The argument below is explicitly presented for $p<N$. The remaining cases are handled as follows.

Case $p=N$: fix an exponent $\kappa>q$ sufficiently large so that its H\"older conjugate satisfies $\kappa'\le r$. The proof carries over by replacing $p^*$ with $\kappa$ and $(p^*)'$ with $\kappa'$, employing the embedding \(W_0^{1,N}(\Omega;\R^m)\hookrightarrow L^\kappa(\Omega;\R^m)\).

Case $p>N$: by the embedding \(W_0^{1,p}(\Omega;\R^m)\hookrightarrow L^\infty(\Omega;\R^m)\), every convergent sequence in $W_0^{1,p}(\Omega;\R^m)$ converges in $L^\infty(\Omega;\R^m)$. By the growth condition \eqref{g.1} together with $a\in L^1(\Omega)$, a direct application of the dominated convergence theorem yields the continuity of $J_2'$.}

}

Let $\{\boldsymbol u_n\}\subset W_0^{1,p}(\Omega; \R^m)$ be such that $\boldsymbol u_n\to\boldsymbol u$ in $W_0^{1,p}(\Omega; \R^m)$. Set $s:=(p^*)'$. Since $q<p^*$, we have $(q-1)s<p^*$. Hence, by Sobolev's embedding,
\[
\boldsymbol u_n\to\boldsymbol u\quad \text{in $L^{(q-1)s}(\Omega;\R^m)$}.
\]
Moreover, up to a subsequence, $\boldsymbol u_n(x)\to\boldsymbol u(x)$ for a.e. $x\in\Omega$ and there exist $c_3>0$ and $\omega_3\in L^1(\Omega)$ such that
\[
|\boldsymbol u_n|^{(q-1)s}\le c_3\omega_3\quad \text{a.e. in $\Omega$}.
\]
By \eqref{g.1},
\[
|\boldsymbol f(x,\boldsymbol u_n)|^s
\le C\left(|a(x)|^s+b^s|\boldsymbol u_n|^{(q-1)s}\right).
\]
Since $a\in L^r(\Omega)$ with $r\ge s=(p^*)'$, the right-hand side is dominated by an $L^1(\Omega)$ function. Therefore, by the dominated convergence theorem,
\[
\boldsymbol f(x,\boldsymbol u_n)\to\boldsymbol f(x,\boldsymbol u)
\quad \text{in $L^{(p^*)'}(\Omega;\R^m)$}.
\]
Finally, for every $\boldsymbol\varphi\in W_0^{1,p}(\Omega;\R^m)$, H\"older's inequality and Sobolev's embedding give
\[
\left|\langle J_2'(\boldsymbol u_n)-J_2'(\boldsymbol u),\boldsymbol\varphi\rangle\right|
\le \left|\boldsymbol f(x,\boldsymbol u_n)-\boldsymbol f(x,\boldsymbol u)\right|_{(p^*)'}|\boldsymbol\varphi|_{p^*}
\le C\left|\boldsymbol f(x,\boldsymbol u_n)-\boldsymbol f(x,\boldsymbol u)\right|_{(p^*)'}\|\boldsymbol\varphi\|.
\]
Thus $J_2'$ is continuous. Consequently, $\mathcal J$ is of class $C^1$, and the proof is complete.
\end{proof}
\begin{remark}
For every $\lambda\in\R$, set
\[
G_\lambda(x,\boldsymbol u):=\frac{\lambda}{p}|\boldsymbol u|^p+F(x,\boldsymbol u).
\]
By Young's inequality, the additional term $\lambda|\boldsymbol u|^{p-2}\boldsymbol u$ still satisfies the growth condition \eqref{g.1}. Hence $G_\lambda$ satisfies the same assumptions as $F$, and Proposition \ref{frechet diff} applies to the perturbed functional
\[
\mathcal J_\lambda(\boldsymbol{u})=\frac1p\int_\Omega|D\boldsymbol{u}|^p-\int_\Omega G_\lambda(x,\boldsymbol u).
\]
In particular, $\mathcal J_\lambda\in C^1$ for every $\lambda\in\R$.
\end{remark}
\begin{remark}
By Proposition \ref{frechet diff}, a function $\boldsymbol u\in W_0^{1,p}(\Omega;\R^m)$ is a critical point of $\mathcal J$ if and only if
\[
\int_\Omega |D\boldsymbol u|^{p-2}D\boldsymbol u\cdot D\boldsymbol\varphi
=\int_\Omega \boldsymbol f(x,\boldsymbol u)\cdot\boldsymbol\varphi,
\qquad \forall\,\boldsymbol\varphi\in W_0^{1,p}(\Omega;\R^m).
\]
Therefore, the critical points of $\mathcal J$ are precisely the weak solutions of \eqref{P}.
\end{remark}

\section{Regularity}\label{regularity section}

In this section, we study the regularity of weak solutions of \eqref{P} when $p\ge2$.\\ We need the following lemma:
{
\begin{lemma}\label{lemma vannella}
Let $p>1$, $r>0$, and $\boldsymbol u_0\in W_0^{1,p}(\Omega;\R^m)$, and set
\[
p^*:=
\begin{cases}
\dfrac{Np}{N-p},& p<N,\\
\infty,& p\ge N.
\end{cases}
\]
If $\varepsilon>0$ and $q\in[1,p^*)$, then there exists $\sigma>0$ such that
\[
\int_{\{|\boldsymbol u|\ge\sigma\}}|\boldsymbol u|^q<\varepsilon
\]
for every $\boldsymbol u\in B_r(\boldsymbol u_0):=\{\boldsymbol u\in W_0^{1,p}(\Omega;\R^m):\|\boldsymbol u-\boldsymbol u_0\|<r\}$.
\end{lemma}

\begin{proof}
Set $R:=r+\|\boldsymbol u_0\|$. If $\boldsymbol u\in B_r(\boldsymbol u_0)$, then the scalar function $w:=|\boldsymbol u|$ belongs to $W_0^{1,p}(\Omega;\R)$ and
\[
\|w\|_{W_0^{1,p}(\Omega;\R)}
\le \|\boldsymbol u\|<R.
\]
If $p<N$, the conclusion follows by applying \cite[Lemma 2.1]{Vannella23} to $w\in B_R(0)$.

If $p=N$, choose $t>q$. By the Sobolev embedding $W_0^{1,N}(\Omega;\R^m)\hookrightarrow L^t(\Omega;\R^m)$, there exists $C_t>0$ such that
\[
\|\boldsymbol u\|_{L^t}\le C_tR
\qquad\text{for every $\boldsymbol u\in B_r(\boldsymbol u_0)$}.
\]
Therefore,
\[
\int_{\{|\boldsymbol u|\ge\sigma\}}|\boldsymbol u|^q
\le \sigma^{q-t}\int_\Omega|\boldsymbol u|^t
\le \sigma^{q-t}(C_tR)^t,
\]
and the right-hand side is smaller than $\varepsilon$ for $\sigma$ sufficiently large.

Finally, if $p>N$, the Sobolev embedding into $L^\infty$ gives a constant $C_\infty>0$ such that
\[
\|\boldsymbol u\|_{L^\infty}\le C_\infty R
\qquad\text{for every $\boldsymbol u\in B_r(\boldsymbol u_0)$}.
\]
Choosing $\sigma>C_\infty R$, the set $\{|\boldsymbol u|\ge\sigma\}$ is empty, and the conclusion follows.
\end{proof}
}

\begin{proof}[Proof of Theorem \ref{regularity cap1}]
We follow the procedure in \cite{Vannella23} (see also \cite{carmona2013regularity} and \cite{guedda1989quasilinear}).
{

The argument below is explicitly presented for $p<N$. The remaining cases are handled as follows.

Case $p=N$: by virtue of the embedding \(W_0^{1,N}(\Omega;\R^m)\hookrightarrow L^\kappa(\Omega;\R^m)\) for every finite $\kappa\ge1$, every occurrence of $p^*$ in the final Stampacchia iteration argument is replaced by a sufficiently large finite exponent $\kappa>p$. Since the preceding iteration gives $f^i(x,\boldsymbol u)\in L^t(\Omega)$ for every finite $t>1$, one can then choose $\theta$ large enough so that the resulting exponent $\gamma$ exceeds one.

Case $p>N$: the boundedness conclusion follows immediately from the embedding \(W_0^{1,p}(\Omega;\R^m)\hookrightarrow L^\infty(\Omega;\R^m)\).
}

For every $\gamma, t, k>1$, we define
\begin{equation*}
h_{k,\gamma}(s) =
\begin{cases}
|s|^{\gamma-1} s, & |s| \leq k, \\
\gamma k^{\gamma-1} s + \operatorname{sign}(s) (1-\gamma)k^\gamma , & |s| > k.
\end{cases}
\end{equation*}
We also define
\begin{equation*}
\Phi_{k,t,\gamma}(s) = \int_0^s|h_{k,\gamma}'(r)|^{\frac{t}{\gamma}}\, dr.
\end{equation*}
{
In particular, 
$h_{k,\gamma}'(s)=\gamma|s|^{\gamma-1}$ when $|s|\le k$ and $h_{k,\gamma}'(s)=\gamma k^{\gamma-1}$ otherwise. Hence, $h_{k,\gamma}$ and $\Phi_{k,t,\gamma}$ are $C^1$ functions with bounded derivatives (depending on $\gamma,t$ and $k$), and both vanish at zero. Therefore, whenever $u^i\in W_0^{1,p}(\Omega)$, the Sobolev chain rule yields
\[
h_{k,\gamma}(u^i),\ \Phi_{k,t,\gamma}(u^i)\in W_0^{1,p}(\Omega),
\qquad
\nabla h_{k,\gamma}(u^i)=h_{k,\gamma}'(u^i)\nabla u^i
\quad\text{a.e. in $\Omega$}.
\]
}
Additionally, for each $t\ge\gamma$, there exists a positive constant $C$, depending on $\gamma$ and $t$ but independent of $k$, such that
\begin{equation}\label{1}
|s|^{\frac{t}{\gamma}-1} |\Phi_{k,t,\gamma}(s)| \leq C |h_{k,\gamma}(s)|^{\frac{t}{\gamma}},
\end{equation}
\begin{equation}\label{2}
|\Phi_{k,t,\gamma}(s)| \leq C |h_{k,\gamma}(s)|^{\frac{1}{\gamma}\left (1 + t \frac{\gamma-1}{\gamma}\right)}.
\end{equation}

We first prove the estimate for the first component; the same argument applies to all other components. Set
\[
\boldsymbol\psi=(\Phi_{k, \gamma p,\gamma}(u^1),0,\dots,0).
\]
Since $\boldsymbol u$ is a weak solution of \eqref{P}, we may use $\boldsymbol\psi$ as a test function. Hence
\begin{equation*}
\langle \mathcal J'(\boldsymbol u), \boldsymbol\psi \rangle =0,
\end{equation*}
for any $k,\gamma > 1$. \\
In view of the fact that $p\ge2$, we can deduce:
\[|\nabla  u^1|^{p-2}\le|D\boldsymbol u|^{p-2}.\]
{
Since $h_{k,\gamma}(u^1)\in W_0^{1,p}(\Omega)$ and  $q<p^*$  , by the continuous Sobolev embedding
\(
W_0^{1,p}(\Omega;\mathbb R)\hookrightarrow L^q(\Omega;\mathbb R)
\)
  we have
\[
\left(\int_\Omega |h_{k,\gamma}(u^1)|^q\right)^{\frac{p}{q}}
\le C_S^p\int_\Omega |\nabla h_{k,\gamma}(u^1)|^p.
\]
Combining this inequality with the subcritical growth assumption \eqref{subcritical growth theorem 1.5}, we can find $c,c_1>0$ such that
}
\begin{equation}\label{stimah}
\begin{aligned}
    \left(\int_\Omega |h_{k,\gamma}( u^1)|^q\right)^{\frac{p}{q}}&\le c\int_\Omega |\nabla h_{k,\gamma}( u^1)|^p=c\int_{\Omega}|\nabla  u^1|^p\cdot|h_{k,\gamma}'( u^1)|^p\\
    &=c\int_{\Omega}|\nabla  u^1|^{p-2}\cdot|\nabla  u^1|^2\cdot|h_{k,\gamma}'( u^1)|^p\\
    &\le c\int_\Omega |D{\boldsymbol{u}}|^{p-2}\cdot|\nabla  u^1|^2\cdot|h_{k,\gamma}'( u^1)|^p\\
    &=c\int_\Omega|D{\boldsymbol{u}}|^{p-2}\nabla  u^1\cdot\nabla\Phi_{k,\gamma p,\gamma}( u^1)=c\int_\Omega f^1(x,\boldsymbol u)\Phi_{k,\gamma p,\gamma}( u^1)\\
    &\le c_1\int_\Omega\left(1+|\boldsymbol u|^{q-1}\right)\left|\Phi_{k,\gamma p,\gamma}( u^1)\right|.
\end{aligned}
\end{equation}

\noindent
Denote
\[ \Omega_{\sigma, \boldsymbol u} = \{ x \in \Omega : |\boldsymbol u(x)| > \sigma \}  .\]
From \eqref{1} and \eqref{2}, we obtain
\begin{align*}
    \int_{\Omega}\left(1+|\boldsymbol u|^{q-1}\right)\left|\Phi_{k,\gamma p,\gamma}( u^1)\right|&=\int_{\Omega_{\sigma,\boldsymbol u}^c}\left(1+|\boldsymbol u|^{q-1}\right)\left|\Phi_{k,\gamma p,\gamma}( u^1)\right|+\int_{\Omega_{\sigma,\boldsymbol u}}\left(1+|\boldsymbol u|^{q-1}\right)\left|\Phi_{k,\gamma p,\gamma}( u^1)\right|\\
    &\le(1+\sigma^{q-1})\int_\Omega\left|\Phi_{k,\gamma p,\gamma}( u^1)\right|+\int_{\Omega_{\sigma,\boldsymbol u}}|\boldsymbol u|^{q-p}\cdot|\boldsymbol u|^{p-1}\left|\Phi_{k,\gamma p,\gamma}( u^1)\right|\\
    &\le(1+\sigma^{q-1})\int_\Omega\left|\Phi_{k,\gamma p,\gamma}( u^1)\right|+C\int_{\Omega_{\sigma,\boldsymbol u}}|\boldsymbol u|^{q-p}\left|h_{k,\gamma}( u^1)\right|^p\\
    &\le C(1+\sigma^{q-1})\int_\Omega\left|h_{k,\gamma}( u^1)\right|^{\frac{p\gamma+1-p}{\gamma}}+C\int_{\Omega_{\sigma,\boldsymbol u}}|\boldsymbol u|^{q-p}\left|h_{k,\gamma}( u^1)\right|^p.
\end{align*}
{
Since
\(
1<\frac{p\gamma+1-p}{\gamma}<p<q,
\)
H\"older's inequality with exponents
$\frac{q\gamma}{p\gamma+1-p}$ and
$\frac{q\gamma}{q\gamma-p\gamma-1+p}$ gives
\[
\int_\Omega |h_{k,\gamma}(u^1)|^{\frac{p\gamma+1-p}{\gamma}}
\le |\Omega|^{1-\frac{p\gamma+1-p}{q\gamma}}
\|h_{k,\gamma}(u^1)\|_{L^q(\Omega;\mathbb R)}^{\frac{p\gamma+1-p}{\gamma}}.
\]
H\"older's inequality with exponents $\frac{q}{q-p}$ and $\frac{q}{p}$ yields
\[
\int_{\Omega_{\sigma,\boldsymbol u}}|\boldsymbol u|^{q-p}
|h_{k,\gamma}(u^1)|^p
\le
\|\boldsymbol u\|_{L^q(\Omega_{\sigma,\boldsymbol u};\mathbb R^m)}^{q-p}
\|h_{k,\gamma}(u^1)\|_{L^q(\Omega;\mathbb R)}^p.
\]
Therefore, absorbing the factor depending on $|\Omega|$ into the constant, we have
}
\begin{align*}
    \int_\Omega\left(1+|\boldsymbol u|^{q-1}\right)\left|\Phi_{k,\gamma p,\gamma}( u^1)\right|&\le C_1(1+\sigma^{q-1})\left\|h_{k,\gamma}( u^1)\right\|_{L^q(\Omega; \R)}^{\frac{\gamma p+1-p}{\gamma}}\\
    &\quad+C_1\|{\boldsymbol{u}}\|_{L^q(\Omega_{\sigma,\boldsymbol u};\R^m)}^{q-p}\cdot\|h_{k,\gamma}( u^1)\|_{L^q(\Omega; \R)}^p.
\end{align*}

Hence, from \eqref{stimah}, we have
\begin{equation}\label{stimah2}
\begin{aligned}
    \|h_{k,\gamma}( u^1)\|_{L^q(\Omega; \R)}^p&\le C_2(1+\sigma^{q-1})\|h_{k,\gamma}( u^1)\|_{L^q(\Omega; \R)}^{\frac{\gamma p+1-p}{\gamma}}\\
    &\quad+C_2\|\boldsymbol u\|_{L^q(\Omega_{\sigma,\boldsymbol u};\R^m)}^{q-p}\left\|h_{k,\gamma}( u^1)\right\|_{L^q(\Omega; \R)}^p.
\end{aligned}
\end{equation}

{
By Young's inequality with conjugate exponents $\frac{1}{1-\eta}$ and $\frac{1}{\eta}$, for every $\eta\in(0,1)$ and $a,b\ge0$, we have
\[
ab^\eta=(4^\eta a)\left(\frac b4\right)^\eta
\le (1-\eta)(4^\eta a)^{\frac{1}{1-\eta}}
+\eta\frac b4
\le (4a)^{\frac{1}{1-\eta}}+\frac b4.
\]
We apply this inequality with
\[
\eta=\frac{\gamma p+1-p}{\gamma p}\in(0,1),
\qquad
a=C_2(1+\sigma^{q-1}),
\qquad
b=\|h_{k,\gamma}(u^1)\|_{L^q(\Omega;\mathbb R)}^p.
\]
Since $\frac{1}{1-\eta}=\frac{\gamma p}{p-1}$, after absorbing the resulting constant into $C_3$, we obtain
}
\begin{align}\label{youngapplied}
    C_2(1+\sigma^{q-1})\left\|h_{k,\gamma}( u^1)\right\|_{L^q(\Omega; \R)}^{\frac{\gamma p+1-p}{\gamma}}\le\frac14\left\|h_{k,\gamma}( u^1)\right\|_{L^q(\Omega; \R)}^{p}+C_3(1+\sigma^{q-1})^{\frac{\gamma p}{p-1}}.
\end{align}

Thus, \eqref{stimah2} and \eqref{youngapplied} yield
\begin{align*}
    \left\|h_{k,\gamma}( u^1)\right\|_{L^q(\Omega; \R)}^{p}\le C_3(1+\sigma^{q-1})^{\frac{\gamma p}{p-1}}+C_2\|\boldsymbol u\|_{L^q(\Omega_{\sigma,\boldsymbol u};\R^m)}^{q-p}\left\|h_{k,\gamma}( u^1)\right\|_{L^q(\Omega; \R)}^p.
\end{align*}
Since $\boldsymbol u\in L^q(\Omega;\R^m)$, Lemma \ref{lemma vannella} gives
\[
\|\boldsymbol u\|_{L^q(\Omega_{\sigma,\boldsymbol u};\R^m)}^{q-p}\to0,\qquad\text{as $\sigma\to+\infty$.}
\]
We choose $\sigma$ large enough so that
\[
C_2\|\boldsymbol u\|_{L^q(\Omega_{\sigma,\boldsymbol u};\R^m)}^{q-p}\le \frac12.
\]
Absorbing the last term into the left-hand side of the previous estimate, we obtain
\begin{equation*}
\int_\Omega | h_{k,\gamma} ({u}^1) |^{q} \quad  \text{ is bounded (uniformly in } k).
\end{equation*}
Similarly,
\begin{equation*}
\int_\Omega | h_{k,\gamma} ({u}^i) |^{q} \quad  \text{ is bounded (uniformly in } k), \qquad i=2,\dots,m.
\end{equation*}

Letting $k\to+\infty$ and using Fatou's lemma, we get $u^i\in L^{\gamma q}(\Omega)$ for every $i=1,\dots,m$. Since $\gamma>1$ is arbitrary, it follows that
\begin{equation*}
 \boldsymbol u \in L^{\tau} (\Omega; \mathbb{R}^{m}) \quad \text{for any } \tau\in(1,\infty).
\end{equation*}
In particular, by the growth assumption \eqref{subcritical growth theorem 1.5},
\[
f^i(x,\boldsymbol u)\in L^t(\Omega)\qquad \text{for every $t>1$ and $i=1,\dots,m$. }
\]
We now prove boundedness by a Stampacchia truncation argument. Let 
    \[\xi_\sigma(s)=
    \begin{cases}
        s-\sigma&s\ge\sigma\\
        0&|s|\le\sigma\\
        s+\sigma&s\le-\sigma.
    \end{cases}
    \]
    Since $\nabla\xi_\sigma(u^i)=\nabla u^i$ on $A_{\sigma,i}:=\{|u^i|\ge\sigma\}$ and vanishes outside this set, Sobolev's embedding and the inequality $|\nabla u^i|^{p-2}\le |D\boldsymbol u|^{p-2}$ give
    \begin{equation}\label{stima1xi}
    \left(\int_\Omega|\xi_\sigma( u^i)|^{p^*}\right)^{\frac{p}{p^*}}\le C\int_{\Omega}|\nabla\xi_\sigma( u^i)|^p\le C\int_{A_{\sigma,i}}|D\boldsymbol u|^{p-2}|\nabla  u^i|^2.\end{equation}
    Moreover, testing \eqref{P} with $\xi_\sigma(u^i)\boldsymbol e_i$, we deduce that
    \begin{equation}
    \int_{A_{\sigma,i}}|D\boldsymbol u|^{p-2}|\nabla u^i|^2=\int_{A_{\sigma,i}} f^i(x,\boldsymbol u)\xi_\sigma(u^i).
    \end{equation}
    By H\"older's inequality, we have:
    \begin{equation}\label{stima2xi}
    \int_{A_{\sigma,i}}|f^i(x,\boldsymbol u)\xi_\sigma(u^i)|\le\left(\int_{A_{\sigma,i}}|f^i(x,\boldsymbol u)|^{(p^*)'}\right)^{\frac{1}{(p^*)'}}\cdot\left(\int_{A_{\sigma,i}}|\xi_\sigma(u^i)|^{p^*}\right)^{\frac{1}{p^*}},\quad \frac{1}{p^*}+\frac{1}{(p^*)'}=1.
    \end{equation}
    We choose $\theta>N/p$. Since $f^i(x,\boldsymbol u)\in L^\theta(\Omega)$, H\"older's inequality on $A_{\sigma,i}$ and \eqref{stima1xi}-\eqref{stima2xi} yield
    \[
    \left(\int_\Omega|\xi_\sigma( u^i)|^{p^*}\right)^{\frac{p-1}{p^*}}\le C\left(\int_{A_{\sigma,i}}|f^i(x,\boldsymbol u)|^{\theta}\right)^{\frac{1}{\theta}}|A_{\sigma,i}|^{\frac{1}{(p^*)'}-\frac1\theta},\quad \theta>\frac{N}{p}.
    \]
    On the other hand, for every $M>\sigma$, we have that
    \[
    \left(\int_\Omega|\xi_\sigma(u^i)|^{p^*} \right)^{\frac{p-1}{p^*}}\ge  \left(\int_{A_{M,i}}|\xi_\sigma(u^i)|^{p^*} \right)^{\frac{p-1}{p^*}}\ge(M-\sigma)^{p-1}|A_{M,i}|^{\frac{p-1}{p^*}}.
    \]
    Then, there exists a positive constant $\tilde C>0$ such that
    \[
    |A_{M,i}|\le\frac{\tilde C}{(M-\sigma)^{p^*}}|A_{\sigma,i}|^\gamma,\qquad \gamma=\frac{p^*}{p-1}\left(\frac{1}{(p^*)'}-\frac1\theta\right)>1.
    \]
    
{
To apply Stampacchia's Lemma \cite[Lemma 4.1]{stampacchia1965probleme}, set
\[
\phi_i(\sigma):=|A_{\sigma,i}|,
\qquad \sigma\ge0.
\]
Since $A_{M,i}\subset A_{\sigma,i}$ whenever $M>\sigma$, the function $\phi_i$ is nonnegative and nonincreasing. Moreover, the preceding estimate reads
\[
\phi_i(M)
\le \frac{\widetilde C}{(M-\sigma)^{p^*}}
[\phi_i(\sigma)]^\gamma
\qquad\text{for every $M>\sigma\ge0$,}
\]
where
\(
\gamma
=\frac{p^*}{p-1}\left(\frac{1}{(p^*)'}-\frac1\theta\right)>1,
\)
the last inequality being equivalent to $\theta>N/p$. This is precisely the hypothesis of case (i) of \cite[Lemma 4.1]{stampacchia1965probleme}, which gives
\(
\phi_i(K)=0,
\) for some $K>0$.
Since $\phi_i$ is nonincreasing, $|A_{\sigma,i}|=\phi_i(\sigma)=0$ for every $\sigma\ge K$. Hence $|u^i|\le K$ a.e. in $\Omega$, and therefore $u^i\in L^\infty(\Omega)$ for every $i=1,\dots,m$. Consequently, $\boldsymbol u\in L^\infty(\Omega;\R^m)$.
}
   
The last part of the statement is then a consequence of the regularity theory up to the boundary, see \cite{benedetto1989boundary}.
\end{proof}

\begingroup

\section{A finite-dimensional decomposition}\label{eigenvalue decomposition section}

We adapt the construction in \cite[Section 5]{candela2009infinitely} to the
vector-valued Sobolev space $W_0^{1,p}(\Omega;\R^m)$.

Set
\[
S_p:=\left\{\boldsymbol u\in W_0^{1,p}(\Omega;\R^m):\int_\Omega|\boldsymbol u|^p=1\right\},
\qquad
I(\boldsymbol u):=\int_\Omega|D\boldsymbol u|^p.
\]
Following \cite[Definition 1.1]{hynd2023uniqueness}, we recall the spectral definition of an eigenfunction.

\begin{definition}[Eigenfunction]\label{vectorial eigenfunction definition}
A nonzero function $\boldsymbol u\in W_0^{1,p}(\Omega;\R^m)$ is an eigenfunction of the vectorial $p$-Laplacian if there exists $\Lambda\in\R$ such that
\begin{equation}\label{vectorial eigenfunction equation}
\int_\Omega |D\boldsymbol u|^{p-2}D\boldsymbol u\cdot D\boldsymbol\varphi
=\Lambda\int_\Omega |\boldsymbol u|^{p-2}\boldsymbol u\cdot\boldsymbol\varphi,
\qquad \forall\,\boldsymbol\varphi\in W_0^{1,p}(\Omega;\R^m).
\end{equation}
The number $\Lambda$ is the corresponding eigenvalue.
\end{definition}

The following proposition makes explicit the equivalence between the
variational and spectral definitions of the first eigenvalue; compare
\cite[Definition 1.1 and Theorem 1.2]{hynd2023uniqueness}.

\begin{proposition}\label{minimum implies eigenfunction}
If $\boldsymbol\varphi_1\in S_p$ minimizes $I$ on $S_p$, then $\boldsymbol\varphi_1$ is an eigenfunction and satisfies
\[
\int_\Omega |D\boldsymbol\varphi_1|^{p-2}D\boldsymbol\varphi_1\cdot D\boldsymbol\psi
=\lambda_1\int_\Omega |\boldsymbol\varphi_1|^{p-2}\boldsymbol\varphi_1\cdot\boldsymbol\psi,
\qquad \forall\,\boldsymbol\psi\in W_0^{1,p}(\Omega;\R^m),
\]
where
\[
\lambda_1=I(\boldsymbol\varphi_1)=\min_{\boldsymbol u\in S_p}I(\boldsymbol u).
\]
Conversely, every eigenfunction associated with $\lambda_1$, normalized in $L^p(\Omega;\R^m)$, is a minimizer of $I$ on $S_p$.
\end{proposition}

\begin{proof}
Since $S_p$ is the level set of the $C^1$ functional
\[
G(\boldsymbol u):=\int_\Omega|\boldsymbol u|^p
\]
and $G'(\boldsymbol\varphi_1)\ne0$, the Lagrange multiplier rule gives a number $\Lambda\in\R$ such that
\[
I'(\boldsymbol\varphi_1)=\Lambda G'(\boldsymbol\varphi_1).
\]
Therefore, for every $\boldsymbol\psi\in W_0^{1,p}(\Omega;\R^m)$,
\[
\int_\Omega |D\boldsymbol\varphi_1|^{p-2}D\boldsymbol\varphi_1\cdot D\boldsymbol\psi
=\Lambda\int_\Omega |\boldsymbol\varphi_1|^{p-2}\boldsymbol\varphi_1\cdot\boldsymbol\psi.
\]
Taking $\boldsymbol\psi=\boldsymbol\varphi_1$ and using $\boldsymbol\varphi_1\in S_p$, we obtain
\[
\Lambda=\int_\Omega|D\boldsymbol\varphi_1|^p=I(\boldsymbol\varphi_1)=\lambda_1.
\]
More generally, if $\boldsymbol u\ne\boldsymbol0$ is any eigenfunction with corresponding eigenvalue $\Lambda$, then testing \eqref{vectorial eigenfunction equation} with $\boldsymbol\varphi=\boldsymbol u$ gives
\begin{equation}\label{eigenvalue above lambda1}
\Lambda
=\frac{\displaystyle\int_\Omega|D\boldsymbol u|^p}
{\displaystyle\int_\Omega|\boldsymbol u|^p}
\ge\lambda_1.
\end{equation}
Hence, the eigenvalue produced by the constrained minimum is the smallest spectral eigenvalue.
Conversely, if $\boldsymbol u\in S_p$ satisfies \eqref{vectorial eigenfunction equation} with $\Lambda=\lambda_1$, then testing with $\boldsymbol\varphi=\boldsymbol u$ gives $I(\boldsymbol u)=\lambda_1$, and hence $\boldsymbol u$ is a minimizer on $S_p$.
\end{proof}

First, we recall the classification of the minimizers of the vectorial Rayleigh quotient proved in \cite[Theorem 1.2]{hynd2023uniqueness}.

\begin{proposition}\label{first vectorial eigenvalue}
Every minimizer $\boldsymbol\varphi_1\in S_p$ of $I$ is synchronized, namely
\(
\boldsymbol\varphi_1=\boldsymbol c\,\varphi_1,
\)
where $\boldsymbol c\in S^{m-1}$ and $\varphi_1\in W_0^{1,p}(\Omega;\R)$ is a normalized scalar first eigenfunction. In particular,
\[
\lambda_1
=\inf\left\{\int_\Omega|\nabla w|^p:\ w\in W_0^{1,p}(\Omega;\R),\ \int_\Omega|w|^p=1\right\}.
\]
Thus, $\lambda_1$ defined in the Introduction is the first eigenvalue of the vectorial $p$-Laplacian and coincides with the first scalar Dirichlet eigenvalue.
\end{proposition}

Now, we construct recursively the variational levels. Let
\[
W_1:=W_0^{1,p}(\Omega;\R^m),
\qquad
\lambda_1:=\inf_{\boldsymbol u\in S_p}I(\boldsymbol u),
\]
and choose a minimizer $\boldsymbol\varphi_1\in S_p$. Once $\boldsymbol\varphi_1,\dots,\boldsymbol\varphi_k$ have been chosen, define the continuous linear functionals
\begin{equation}\label{functionals Lk}
L_j(\boldsymbol u):=\int_\Omega
|\boldsymbol\varphi_j|^{p-2}\boldsymbol\varphi_j\cdot\boldsymbol u,
\qquad j=1,\dots,k.
\end{equation}
Since $\|\boldsymbol\varphi_j\|_{L^p}=1$, H\"older's inequality gives
\[
\|L_j\|_{(L^p)' }=1,
\qquad
L_j(\boldsymbol\varphi_j)=1.
\]
For $k\ge1$, set
\begin{equation}\label{spaces Wk}
W_{k+1}:=\bigcap_{j=1}^{k}\ker L_j
\end{equation}
and define
\begin{equation}\label{lambdak}
\lambda_{k+1}:=\inf\left\{I(\boldsymbol u):\boldsymbol u\in S_p\cap W_{k+1}\right\}.
\end{equation}

\begin{definition}[Pseudo-eigenfunction]\label{pseudo eigenfunction definition}
For $k\ge1$, a pseudo-eigenfunction associated with the variational level $\lambda_{k+1}$ is any function $\boldsymbol\varphi_{k+1}\in S_p\cap W_{k+1}$ such that
\(
I(\boldsymbol\varphi_{k+1})=\lambda_{k+1}.
\)
\end{definition}

{
\begin{remark}
By the Lagrange multiplier rule, for every pseudo-eigenfunction $\boldsymbol\varphi_{k+1}$ there exist $\eta_1,\dots,\eta_k\in\R$ such that
\begin{equation}\label{lagrange identity pseudo eigenfunctions}
\begin{aligned}
\int_\Omega |D\boldsymbol\varphi_{k+1}|^{p-2}
D\boldsymbol\varphi_{k+1}\cdot D\boldsymbol\psi
=\lambda_{k+1}\int_\Omega
|\boldsymbol\varphi_{k+1}|^{p-2}
\boldsymbol\varphi_{k+1}\cdot\boldsymbol\psi+\sum_{j=1}^k\eta_j
\int_\Omega|\boldsymbol\varphi_j|^{p-2}
\boldsymbol\varphi_j\cdot\boldsymbol\psi
\end{aligned}
\end{equation}
for every $\boldsymbol\psi\in W_0^{1,p}(\Omega;\R^m)$. If $p=2$, the functionals $L_j$  reduce to $L^2$-scalar products with $\boldsymbol\varphi_j$. By virtue of the $L^2$-orthogonality of the previously constructed functions, an inductive argument yields  $\eta_1=\dots=\eta_k=0$.
Consequently for $p=2$ every pseudo-eigenfunction is an eigenfunction of the vectorial Laplacian. If $p\ne2$, the constraints are no longer induced by a symmetric scalar product and the additional multipliers need not vanish. Therefore, a pseudo-eigenfunction is not in general an eigenfunction of the unconstrained vectorial $p$-Laplacian.
\end{remark}
}

\begin{lemma}\label{existence pseudo eigenfunctions}
For every $k\ge1$, the level $\lambda_k$ is achieved by some $\boldsymbol\varphi_k\in S_p\cap W_k$.
\end{lemma}

\begin{proof}
Let $\{\boldsymbol u_n\}\subset S_p\cap W_k$ be a minimizing sequence. It is bounded in $W_0^{1,p}(\Omega;\R^m)$, so, up to a subsequence,
\[
\boldsymbol u_n\rightharpoonup\boldsymbol u\quad\text{in $W_0^{1,p}(\Omega;\R^m)$},
\qquad
\boldsymbol u_n\to\boldsymbol u\quad\text{in $L^p(\Omega;\R^m)$}.
\]
The strong $L^p$ convergence yields $\boldsymbol u\in S_p$ and, since every $L_j$ is continuous on $L^p(\Omega;\R^m)$, it also yields $\boldsymbol u\in W_k$. Finally, the weak lower semicontinuity of $I$ gives
\[
I(\boldsymbol u)\le\liminf_{n\to\infty}I(\boldsymbol u_n)=\lambda_k.
\]
Thus $I(\boldsymbol u)=\lambda_k$, and we may take $\boldsymbol\varphi_k=\boldsymbol u$.
\end{proof}

{
\begin{remark}
If $p\ge2$, every pseudo-eigenfunction $\boldsymbol\varphi_k$ belongs to $L^\infty(\Omega;\R^m)$. We prove this assertion  by induction, applying  Theorem \ref{regularity cap1}. For $k=1$, the Euler--Lagrange equation is
\[
-\boldsymbol\Delta_p\boldsymbol\varphi_1
=\lambda_1|\boldsymbol\varphi_1|^{p-2}\boldsymbol\varphi_1
\]
and Proposition \ref{first vectorial eigenvalue} yields $\boldsymbol\varphi_1=\boldsymbol c\,\varphi_1$ for some $\boldsymbol c\in S^{m-1}$ and a normalized scalar first eigenfunction $\varphi_1$. Since \cite[Lemma 5.1]{candela2009infinitely} ensures $\varphi_1\in L^\infty(\Omega;\mathbb R)$, it follows that $\boldsymbol\varphi_1\in L^\infty(\Omega;\mathbb R^m)$.
Assume now that $\boldsymbol\varphi_1,\dots,\boldsymbol\varphi_k\in L^\infty(\Omega;\R^m)$. 
Then the Euler--Lagrange equation reads
\[
-\boldsymbol\Delta_p\boldsymbol\varphi_{k+1}
=\lambda_{k+1}|\boldsymbol\varphi_{k+1}|^{p-2}\boldsymbol\varphi_{k+1}
+\boldsymbol g_k(x),
\qquad
\boldsymbol g_k:=\sum_{j=1}^k\eta_j
|\boldsymbol\varphi_j|^{p-2}\boldsymbol\varphi_j\in L^\infty(\Omega;\R^m).
\]
Since $\boldsymbol g_k\in L^\infty(\Omega;\R^m)$, the right-hand side exhibits a $(p-1)$-growth in $\boldsymbol\varphi_{k+1}$.
Choosing any \(q\in(p,p^*)\), we have
\(
|s|^{p-1}\le 1+|s|^{q-1},
\)
so that  assumption \eqref{subcritical growth theorem 1.5} of Theorem \ref{regularity cap1} is satisfied. Hence, $\boldsymbol\varphi_{k+1}\in L^\infty(\Omega;\R^m)$. Therefore, for $p\ge2$ every pseudo-eigenfunction belongs to $L^\infty(\Omega;\R^m)$. If, in addition, $\partial\Omega\in C^{1,\alpha}$ with $\alpha\in(0,1)$, then \cite{benedetto1989boundary} yields $\boldsymbol\varphi_k\in C^{1,\beta_k}(\overline\Omega;\R^m)$ for some $\beta_k\in(0,1)$ and every $k\ge1$.
\end{remark}
}

By construction,
\begin{equation}\label{triangular orthogonality}
L_i(\boldsymbol\varphi_i)=1,
\qquad
L_i(\boldsymbol\varphi_j)=0\quad\text{whenever $i<j$}.
\end{equation}
Moreover, $W_{k+1}\subset W_k$, and hence
\begin{equation}\label{monotonicity variational levels}
0<\lambda_1\le\lambda_2\le\dots\le\lambda_k\le\dots.
\end{equation}
{
\begin{lemma}\label{first repeated variational levels}
Let $\varphi_1$ be a normalized positive scalar first eigenfunction. Then there exist orthonormal vectors $\boldsymbol c_1,\dots,\boldsymbol c_m\in\R^m$ such that
\[
\boldsymbol\varphi_j=\boldsymbol c_j\varphi_1,
\quad j=1,\dots,m,
\qquad\text{and}\qquad
\lambda_1=\dots=\lambda_m<\lambda_{m+1}.
\]
Moreover, the sequence $\{\lambda_k\}$ diverges to $+\infty$.
\end{lemma}

\begin{proof}
We argue by induction. The assertion for $k=1$ follows from Proposition \ref{first vectorial eigenvalue}. Suppose that, for some $k<m$, one has $\lambda_1=\dots=\lambda_k$ and
\[
\boldsymbol\varphi_j=\boldsymbol c_j\varphi_1,
\qquad j=1,\dots,k,
\]
for some orthonormal vectors $\boldsymbol c_1,\dots,\boldsymbol c_k$. Choose $\boldsymbol d\in S^{m-1}$ orthogonal to $\boldsymbol c_1,\dots,\boldsymbol c_k$. Then $\boldsymbol d\varphi_1\in S_p\cap W_{k+1}$ and
\[
I(\boldsymbol d\varphi_1)=\lambda_1.
\]
Hence, $\lambda_{k+1}\le\lambda_1$, while \eqref{monotonicity variational levels} gives the opposite inequality. Thus, $\lambda_{k+1}=\lambda_1$. Proposition \ref{first vectorial eigenvalue} now yields
\[
\boldsymbol\varphi_{k+1}=\boldsymbol c_{k+1}\varphi_1
\]
for some $\boldsymbol c_{k+1}\in S^{m-1}$. Moreover, since $\boldsymbol\varphi_{k+1}\in W_{k+1}$, for every $j=1,\dots,k$ we have
\[
0=L_j(\boldsymbol\varphi_{k+1})
=\int_\Omega |\boldsymbol c_j\varphi_1|^{p-2}
\boldsymbol c_j\varphi_1\cdot\boldsymbol c_{k+1}\varphi_1
=\boldsymbol c_j\cdot\boldsymbol c_{k+1}.
\]
This proves both the synchronization and the equality $\lambda_1=\dots=\lambda_m$.

Finally, if $\lambda_{m+1}=\lambda_1$, then the minimizer $\boldsymbol\varphi_{m+1}$ given by Lemma \ref{existence pseudo eigenfunctions} would again be synchronized by Proposition \ref{first vectorial eigenvalue}, so that $\boldsymbol\varphi_{m+1}=\boldsymbol c\varphi_1$ for some $\boldsymbol c\in S^{m-1}$. Since $\boldsymbol\varphi_{m+1}\in W_{m+1}$, we would have
\[
0=L_j(\boldsymbol\varphi_{m+1})
=\boldsymbol c_j\cdot\boldsymbol c,
\qquad j=1,\dots,m.
\]
However, $\boldsymbol c_1,\dots,\boldsymbol c_m$ form an orthonormal basis of $\R^m$, and hence $\boldsymbol c=\boldsymbol0$, a contradiction. Therefore, $\lambda_m<\lambda_{m+1}$.

It remains to prove the divergence of the levels.
Assume by contradiction that $\lambda_k\le\overline\lambda$ for every $k$. Then $\{\boldsymbol\varphi_k\}$ is bounded in $W_0^{1,p}(\Omega;\R^m)$. By the compact embedding of $W_0^{1,p}(\Omega;\R^m)$ into $L^p(\Omega;\R^m)$, a subsequence, still denoted by $\{\boldsymbol\varphi_k\}$, is Cauchy in $L^p(\Omega;\R^m)$. On the other hand, if $i<j$, then \eqref{triangular orthogonality} implies
\[
1=\left|L_i(\boldsymbol\varphi_i-\boldsymbol\varphi_j)\right|
\le \|L_i\|_{(L^p)'}
\|\boldsymbol\varphi_i-\boldsymbol\varphi_j\|_{L^p}
=\|\boldsymbol\varphi_i-\boldsymbol\varphi_j\|_{L^p},
\]
which contradicts the Cauchy property.
\end{proof}
}

For $k\ge1$, define
\(
V_k:=\operatorname{span}\left\{\boldsymbol\varphi_1,\dots,
\boldsymbol\varphi_k\right\}.
\)

\begin{lemma}\label{finite dimensional decomposition}
For every $k\ge1$,
\begin{equation}\label{decomposition}
W_0^{1,p}(\Omega;\R^m)=V_k\oplus W_{k+1}.
\end{equation}
In particular, $W_{k+1}$ is a closed subspace of codimension $k$.
\end{lemma}

\begin{proof}
Let $\boldsymbol v=\sum_{i=1}^k\beta_i\boldsymbol\varphi_i\in V_k$. By \eqref{triangular orthogonality},
\[
L_1(\boldsymbol v)=\beta_1
\]
and, for $j=2,\dots,k$,
\[
L_j(\boldsymbol v)=\sum_{i=1}^{j-1}
\beta_iL_j(\boldsymbol\varphi_i)+\beta_j.
\]
If $\boldsymbol v\in W_{k+1}$, then $L_j(\boldsymbol v)=0$ for every $j=1,\dots,k$. Since $L_1(\boldsymbol v)=\beta_1$, we first obtain $\beta_1=0$. Assuming that $\beta_1=\dots=\beta_{j-1}=0$, the preceding identity gives $L_j(\boldsymbol v)=\beta_j$, and hence $\beta_j=0$. Therefore, $\beta_1=\dots=\beta_k=0$, so $V_k\cap W_{k+1}=\{\boldsymbol0\}$.

Now fix $\boldsymbol u\in W_0^{1,p}(\Omega;\R^m)$ and define
\[
\beta_1:=L_1(\boldsymbol u),
\qquad
\beta_j:=L_j(\boldsymbol u)-\sum_{i=1}^{j-1}
\beta_iL_j(\boldsymbol\varphi_i),
\quad j=2,\dots,k.
\]
Set
\[
\boldsymbol v:=\sum_{i=1}^k\beta_i\boldsymbol\varphi_i,
\qquad
\boldsymbol w:=\boldsymbol u-\boldsymbol v.
\]
The definition of the coefficients and \eqref{triangular orthogonality} imply $L_j(\boldsymbol w)=0$ for $j=1,\dots,k$. Thus $\boldsymbol w\in W_{k+1}$, proving \eqref{decomposition}.
\end{proof}

\begin{proposition}\label{complete pseudo eigenfunctions}
For every $k\ge1$,
\begin{equation}\label{coercivity Wk}
\lambda_{k+1}\int_\Omega|\boldsymbol w|^p
\le\int_\Omega|D\boldsymbol w|^p
\qquad\text{for every $\boldsymbol w\in W_{k+1}$}.
\end{equation}
Moreover, the pseudo-eigenfunctions $\{\boldsymbol\varphi_k\}$ generate the whole space $W_0^{1,p}(\Omega;\R^m)$. 
\end{proposition}

\begin{proof}
Let $\boldsymbol w\in W_{k+1}\setminus\{\boldsymbol0\}$. Then
\[
\frac{\boldsymbol w}{\|\boldsymbol w\|_{L^p}}
\in S_p\cap W_{k+1}.
\]
The definition of $\lambda_{k+1}$ therefore gives \eqref{coercivity Wk}; the case $\boldsymbol w=\boldsymbol0$ is immediate.

Finally, if
\(
\boldsymbol w\in \displaystyle\bigcap_{k\ge1}\ker L_k=\bigcap_{k\ge1}W_k,
\)
then \eqref{coercivity Wk} gives
\[
\lambda_{k+1}\int_\Omega|\boldsymbol w|^p
\le\int_\Omega|D\boldsymbol w|^p
\qquad\text{for every $k\ge1$}.
\]
Since $\lambda_k\to+\infty$, it follows that $\boldsymbol w=\boldsymbol0$.
\end{proof}

{
\begin{remark}
We briefly explain the meaning of the statement that the
pseudo-eigenfunctions $\{\boldsymbol\varphi_k\}$ generate
$W_0^{1,p}(\Omega;\R^m)$. Define
\[
A=-\boldsymbol\Delta_p:W_0^{1,p}(\Omega;\R^m)\to W^{-1,p'}(\Omega;\R^m),
\qquad
\langle A(\boldsymbol u),\boldsymbol v\rangle
:=\int_\Omega |D\boldsymbol u|^{p-2}D\boldsymbol u\cdot D\boldsymbol v.
\]
Adapting the scalar argument in  \cite[Theorems 7--9]{dinca2001variational} to the vectorial setting 
 we prove that
$A:W_0^{1,p}(\Omega;\R^m)\to W^{-1,p'}(\Omega;\R^m)$ is a homeomorphism.

{
We
record the essential details.
Setting
\(
\Psi(\boldsymbol u):=\frac1p\|\boldsymbol u\|^p,
\)
we have $\Psi\in C^1(W_0^{1,p}(\Omega;\R^m);\R)$ with $\Psi'=A$. Moreover,
\(
\langle A(\boldsymbol u),\boldsymbol u\rangle=\|\boldsymbol u\|^p
\) and
\(\|A(\boldsymbol u)\|_{W^{-1,p'}(\Omega;\R^m)}=\|\boldsymbol u\|^{p-1}.
\)
Thus, $A$ is continuous, coercive and strictly monotone, the latter property holding because
the map $\xi\mapsto|\xi|^{p-2}\xi$ is strictly monotone on
$\R^{m\times N}$. By the Browder--Minty theorem, $A$ is
surjective, while strict monotonicity gives injectivity.

For completeness, the continuity of $A^{-1}$ follows from the vectorial
monotonicity inequalities \eqref{vectorial inequality1}--\eqref{vector inequality2}, which imply that if
$\boldsymbol u_n\rightharpoonup\boldsymbol u$ in $W_0^{1,p}(\Omega;\R^m)$ and
\[
\limsup_{n\to\infty}
\langle A(\boldsymbol u_n)-A(\boldsymbol u),
\boldsymbol u_n-\boldsymbol u\rangle\le0,
\]
then $\boldsymbol u_n\to\boldsymbol u$ strongly in $W_0^{1,p}(\Omega;\R^m)$.

Indeed, if $\boldsymbol f_n\to\boldsymbol f$ in
$W^{-1,p'}(\Omega;\R^m)$ and
$\boldsymbol u_n=A^{-1}(\boldsymbol f_n)$, then $\{\boldsymbol u_n\}$ is
bounded in $W_0^{1,p}(\Omega;\R^m)$,
and every weakly convergent
subsequence, say $\boldsymbol u_n\rightharpoonup\boldsymbol u$, satisfies
\[
\limsup_{n\to\infty}
\langle A(\boldsymbol u_n),\boldsymbol u_n-\boldsymbol u\rangle=0.
\]
This yields
$\boldsymbol u_n\to\boldsymbol u$ in $W_0^{1,p}(\Omega;\R^m)$, and the continuity of $A$ gives
$A(\boldsymbol u)=\boldsymbol f$. By uniqueness, the whole sequence
$A^{-1}(\boldsymbol f_n)$ converges to $A^{-1}(\boldsymbol f)$, proving that
$A^{-1}$ is continuous and hence that $A$ is a homeomorphism.
}

The proof of Proposition \ref{complete pseudo eigenfunctions} gives
$\bigcap_{k\ge1}\ker L_k=\{\boldsymbol0\}$, and hence $\{L_k\}$ generates
the dual space $W^{-1,p'}(\Omega;\R^m)$. By the Lagrange identity
\eqref{lagrange identity pseudo eigenfunctions}, the family
$\{A(\boldsymbol\varphi_k)\}$ also generates this dual space.
Since $A^{-1}$ is a homeomorphism, it follows that
\[
A^{-1}\!\left(\operatorname{span}\{A(\boldsymbol\varphi_k):k\ge1\}\right)
\quad\text{is dense in }W_0^{1,p}(\Omega;\R^m).
\]
When $p=2$, $A:W_0^{1,2}(\Omega;\R^m)\to W^{-1,2}(\Omega;\R^m)$ is a
linear isomorphism, so the family $\{\boldsymbol\varphi_k\}$ generates
$W_0^{1,2}(\Omega;\R^m)$ in the classical linear sense.
\end{remark}
}

\endgroup

\section{Existence results}

\subsection{The \texorpdfstring{$p$}{p}-sublinear case}

\begin{proof}[Proof of Theorem \ref{thm 1}]
    {
    The proof is presented explicitly for $p<N$. The remaining cases are addressed as follows.

    Case $p=N$: choose a finite exponent $\kappa>\max\{p,q\}$ with its H\"older conjugate exponent satisfying $\kappa'\le r$ and use the embedding \(W_0^{1,N}(\Omega;\R^m)\hookrightarrow L^\kappa(\Omega;\R^m)\). The estimates carry over by replacing $p^*$ with $\kappa$ in all H\"older and Sobolev inequalities.

    Case $p>N$: use the embedding \(W_0^{1,p}(\Omega;\R^m)\hookrightarrow L^\infty(\Omega;\R^m)\).

 In particular, $\int_\Omega a(x)|\boldsymbol u|\le\|a\|_{L^1}\|\boldsymbol u\|_{L^\infty}$, and the remaining terms are estimated analogously.
    }

    From \eqref{g.1}, we can deduce 
     \begin{align*}
    \mathcal J(\boldsymbol u)&\ge\frac{1}{p}\int_\Omega |D\boldsymbol{u}|^p-\int_\Omega a(x)|\boldsymbol{u}|-b\int_\Omega|\boldsymbol{u}|^q.
    \end{align*}
    Moreover, by H\"older and Sobolev's inequalities, we have 
    \begin{align*}
     \mathcal J(\boldsymbol u)&\ge\frac{1}{p}\int_\Omega |D\boldsymbol{u}|^p-|a|_{\frac{p^*}{p^*-1}}\cdot|\boldsymbol u|_{p^*}-b\int_\Omega|\boldsymbol{u}|^q\\
     &\ge\frac1p\|\boldsymbol{u}\|^p-C_1\|\boldsymbol{u}\|-C_2\|\boldsymbol{u}\|^q\to+\infty,\ \ \text{as $\|\boldsymbol{u}\|\to+\infty$}.
    \end{align*}
    Hence, $\mathcal J$ is coercive.\\
    Now, let $\{\boldsymbol{u}_n\}\subset W_0^{1,p}(\Omega; \R^m)$ be a minimizing sequence:
    \[\mathcal J(\boldsymbol{u}_n)=\inf_{\boldsymbol{u}\in W_0^{1,p}(\Omega; \R^m)}\mathcal J(\boldsymbol{u})+o(1).\]
     Since $\mathcal J$ is coercive, $\{\boldsymbol{u}_n\}$ is bounded in $W_0^{1,p}(\Omega; \R^m)$. Up to a subsequence, we have 
    $\boldsymbol{u}_n\wto\boldsymbol{u}$ in $W_0^{1,p}(\Omega; \R^m)$.
    
   Therefore, by exploiting Lebesgue's Theorem and possibly passing to a further subsequence, we obtain
    \[\lim_{n\to+\infty}\int_\Omega F(x,\boldsymbol{u}_n)=\int_\Omega F(x,\boldsymbol{u}),\qquad \mathcal J(\boldsymbol{u})\le\liminf_{n\to+\infty}\mathcal J(\boldsymbol{u}_n)=\inf_{\boldsymbol{v}\in W_0^{1,p}(\Omega; \R^m)}\mathcal J(\boldsymbol{v}).\]
    Hence, $\boldsymbol{u}$ is a minimizer.
    
    From \eqref{g.3}, we have there exist  $\beta>\lambda_{1}$ and $\delta>0$ such that 
    \[F(x,\boldsymbol{s})\ge\frac{\beta}{p}|\boldsymbol{s}|^p,\ \ |\boldsymbol{s}|\le\delta.\]
     Let $\varepsilon>0$ be small enough and let $\boldsymbol{\varphi}_{1}$ be a normalized minimizer associated with $\lambda_{1}$. By Proposition \ref{first vectorial eigenvalue}, it is a first eigenfunction of the vectorial $p$-Laplacian. Hence,
    \begin{align*}
        \mathcal J(\varepsilon\boldsymbol{\varphi}_1)&\le\frac{\varepsilon^p}{p}\int_\Omega |D\boldsymbol{\varphi}_1|^p-\frac{\beta\varepsilon^p}{p}\int_\Omega|\boldsymbol{\varphi}_1|^p\\
        &=\frac{(\lambda_1-\beta)\varepsilon^p}{p}\int_\Omega|\boldsymbol{\varphi}_1|^p<0.
    \end{align*}
    Thus, $\mathcal J(\boldsymbol{u})\le\mathcal J(\varepsilon\boldsymbol{\varphi}_1)<0$, and the minimizer is not trivial.
\end{proof}
\begin{proposition}\label{uniqueness cap1}
  Let $1<q<p$ and assume that \eqref{g.1} holds.
Suppose  that the functional 
    \[
        W_0^{1,p}(\Omega;\mathbb{R}^m) \ni \boldsymbol{u} \mapsto -\int_\Omega F(x, \boldsymbol{u})
    \]
    is convex. Then, there exists exactly one weak solution of \eqref{P}.
\end{proposition}

\begin{proof}
Reasoning as in the proof of Theorem \ref{thm 1}, one can show the existence of a minimizer for $\mathcal J$.

    Since the map \( \boldsymbol{u} \mapsto \|\boldsymbol{u}\|^p \) is strictly convex on \( W_0^{1,p}(\Omega; \mathbb{R}^m) \) for \( p > 1 \), it follows that \( \mathcal{J} \) is the sum of a strictly convex functional and a convex functional. Hence, \( \mathcal{J} \) is strictly convex. Therefore, the minimizer is unique, and the proof is complete.
\end{proof}

\begin{corollary}
    The problem
    \[
        \begin{cases}
            -\boldsymbol{\Delta}_p \boldsymbol{u} = \boldsymbol{f} & \text{in } \Omega, \\
            \boldsymbol{u} = 0 & \text{on } \partial\Omega
        \end{cases}
    \]
    has exactly one weak solution for every \( \boldsymbol{f} \in W^{-1,p'}(\Omega; \mathbb{R}^m) \).
\end{corollary}

\begin{proof}
    The associated energy functional $ \mathcal{J} $ is given by
    \[
        \mathcal{J}(\boldsymbol{u}) = \frac{1}{p} \|\boldsymbol{u}\|^p - \langle \boldsymbol{f}, \boldsymbol{u} \rangle.
    \]
    The second term is linear and, in particular, convex. The thesis follows from Proposition~\ref{uniqueness cap1}.
\end{proof}
\begin{remark}
   {We note that $\boldsymbol{f}\in W^{-1,p'}(\Omega;\mathbb{R}^m)$ if, for instance, $\boldsymbol{f}\in L^r(\Omega;\R^m)$, with $r$ satisfying \eqref{r assumption}.}
\end{remark}
\subsection{The \texorpdfstring{$p$}{p}-superlinear case}
\begin{proof}[Proof of Theorem \ref{thm 2}]
   {We apply Theorem \ref{MP}, due to Ambrosetti and Rabinowitz, by verifying its three assumptions.}
  
 {
   The proof is presented for $p<N$. The remaining cases are handled as follows.

   Case $p=N$: fix a finite exponent $\kappa>\max\{q,\mu\}$, with its H\"older conjugate exponent satisfying $\kappa'\le r$, and use the embedding \(W_0^{1,N}(\Omega;\R^m)\hookrightarrow L^\kappa(\Omega;\R^m)\). We replace $p^*$ with $\kappa$ throughout the interpolation and compactness arguments below.

   Case $p>N$: use the embedding \(W_0^{1,p}(\Omega;\R^m)\hookrightarrow L^\infty(\Omega;\R^m)\). In Step 1, the interpolation inequality is replaced with
   $\int_\Omega|\boldsymbol u|^q\le\|\boldsymbol u\|_{L^p}^p\|\boldsymbol u\|_{L^\infty}^{q-p}$. In Step 3, bounded sequences are compact in every $L^s(\Omega;\R^m)$ space with $s\ne\infty$.
   }
   We divide the proof into three steps:\\
  \textbf{Step 1}: we prove that there exist $\alpha,\rho>0$ and a finite-codimensional subspace $W\subset W_0^{1,p}(\Omega; \R^m)$ such that $\mathcal J\ge\alpha$ on $\partial B_\rho\cap W$.
  
   {We consider $W_{h_0}$ defined in \eqref{spaces Wk}, and we will choose $h_0\in\N$ later.}
{
   Let $\rho>0$ and $\boldsymbol{u}\in\partial B_\rho\cap W_{h_0}$. First, \eqref{g.1} and $F(x,\boldsymbol0)=0$ imply
   \[
   |F(x,\boldsymbol s)|\le a(x)|\boldsymbol s|+\frac{b}{q}|\boldsymbol s|^q.
   \]
   Therefore, by H\"older's inequality and the Sobolev embedding, using $r'\le p^*$ as a consequence of \eqref{r assumption}, there exist constants $C_1,C_2>0$ such that
   \[
   \mathcal J(\boldsymbol u)
   \ge \frac1p\|\boldsymbol u\|^p
   -C_1\int_\Omega|\boldsymbol u|^q
   -C_2\|\boldsymbol u\|.
   \]
   Next, choose $t\in(0,p)$ so that
   \(
   \frac{t}{p}+\frac{q-t}{p^*}=1.
   \)
   Applying H\"older's inequality to
   $|\boldsymbol u|^q=|\boldsymbol u|^t|\boldsymbol u|^{q-t}$ gives
   \[
   \int_\Omega|\boldsymbol u|^q
   \le
   \left(\int_\Omega|\boldsymbol u|^p\right)^{\frac{t}{p}}
   \left(\int_\Omega|\boldsymbol u|^{p^*}\right)^{\frac{q-t}{p^*}}.
   \]
   Since $\boldsymbol u\in W_{h_0}$, estimate \eqref{coercivity Wk} yields
   \[
   \left(\int_\Omega|\boldsymbol u|^p\right)^{\frac{t}{p}}
   \le \lambda_{h_0}^{-\frac{t}{p}}\|\boldsymbol u\|^t.
   \]
   Moreover, Sobolev's inequality gives
   \[
   \left(\int_\Omega|\boldsymbol u|^{p^*}\right)^{\frac{q-t}{p^*}}
   \le C\|\boldsymbol u\|^{q-t}.
   \]
   Combining the preceding estimates and absorbing the Sobolev constant into $C_1$, we obtain
   \[
   \mathcal J(\boldsymbol u)
   \ge \frac1p\|\boldsymbol u\|^p
   -C_1\lambda_{h_0}^{-\frac{t}{p}}\|\boldsymbol u\|^q
   -C_2\|\boldsymbol u\|.
   \]
   Since $\|\boldsymbol u\|=\rho$, this becomes
   \[
   \mathcal J(\boldsymbol u)
   \ge \rho\left[
   \left(\frac1p-C_1\lambda_{h_0}^{-\frac{t}{p}}\rho^{q-p}\right)
   \rho^{p-1}-C_2
   \right].
   \]
   We now fix
   \[
   \rho:=\bigl[2p(1+C_2)\bigr]^{\frac1{p-1}}.
   \]
   Since $\lambda_{h_0}\to+\infty$, we can choose $h_0$ sufficiently large that
   \[
   \frac1p-C_1\lambda_{h_0}^{-\frac{t}{p}}\rho^{q-p}
   \ge\frac1{2p}.
   \]
   With these choices, the expression in square brackets is bounded below by $1$, and hence
   \[
   \mathcal J(\boldsymbol u)\ge\rho
   \qquad\text{for every }\boldsymbol u\in\partial B_\rho\cap W_{h_0}.
   \]
}
 \textbf{Step 2}: we prove that, for every finite-dimensional subspace $V\subset W_0^{1,p}(\Omega; \R^m)$, there exists $R>\rho$ such that $\mathcal J\le 0$ on $B_R^c\cap V$.
 
  Let $\boldsymbol\varphi\in V$ with $\|\boldsymbol\varphi\|=1$. By \eqref{g.2},  we have that 
    \[F(x,\boldsymbol s)\ge b_0(x)|\boldsymbol s|^{\mu}-a_0(x),\ \ \text{with $a_0,b_0\in L^1(\Omega)$, $b_0(x)>0$ a.e. in $\Omega$}.\]
   Therefore,
    \begin{align*}
        \mathcal J(t\boldsymbol\varphi)\le \frac{Ct^{p}}{p}-t^{\mu}\int_\Omega b_0(x)|\boldsymbol\varphi|^{\mu}+\int_\Omega|a_0(x)|\to-\infty,\qquad \text{ as $t\to+\infty$.}
    \end{align*}
   Thus, the claim follows by taking $R>0$ large enough. \\
%
    \textbf{Step 3}: we prove that $\mathcal J$ satisfies the $(PS)_c$-condition.
    
      Let $\{\boldsymbol{u}_n\}$ be a Palais-Smale sequence at level $c\in\R$, then $\{\mathcal J(\boldsymbol u_n)\}$ is bounded in $\R$ and 
    \[\int_\Omega|D\boldsymbol{u}_n|^{p-2}D\boldsymbol{u}_n\cdot D\boldsymbol{\varphi}-\int_\Omega \boldsymbol{f}(x,\boldsymbol{u}_n)\cdot\boldsymbol{\varphi}={o(1)}.\]
   Hence, there exists $C>0$ such that
    \begin{align*}
        C+1+\|\boldsymbol{u}_n\|&\ge\mathcal J(\boldsymbol{u}_n)-\frac{1}{\mu}\langle\mathcal J'(\boldsymbol{u}_n),\boldsymbol{u}_n\rangle\\
        &=\left(\frac1p-\frac{1}{\mu}\right)\int_\Omega |D\boldsymbol{u}_n|^p-\int_\Omega\left[F(x,\boldsymbol{u}_n)-\frac{1}{\mu}\boldsymbol{f}(x,\boldsymbol{u}_n)\cdot\boldsymbol{u}_n\right].
    \end{align*}
    According to \eqref{g.2}, we have that there exists $C_1>0$ such that
    \[
    C+1+\|\boldsymbol u_n\|\ge\left(\frac1p-\frac{1}{\mu}\right)\|\boldsymbol{u}_n\|^p-C_1.
    \]  
  Since $\mu>p$, we have that $\{\boldsymbol{u}_n\}$ is bounded in $W_0^{1,p}(\Omega; \R^m)$. Then, there exists a subsequence weakly convergent at some $\boldsymbol u\in W_0^{1,p}(\Omega; \R^m)$ in $W_0^{1,p}(\Omega; \R^m)$ and strongly convergent in $L^q(\Omega; \R^m)$ for every $q\in[1,p^*)$. In particular:
    \[\lim_{n\to+\infty}\int_\Omega f^j(x,\boldsymbol u_n)u_n^j=\int_{\Omega}f^j(x,\boldsymbol u)u^j,\ \ \forall\ j=1,\dots,m,\]
    which implies that 
    \[\lim_{n\to+\infty}\int_\Omega [\boldsymbol f(x,\boldsymbol u_n)-\boldsymbol f(x,\boldsymbol u)]\cdot(\boldsymbol u_n-\boldsymbol u)=0.\]
    \noindent
    Since \(\boldsymbol u_n-\boldsymbol u\rightharpoonup0\) in \(W^{1,p}_0(\Omega;\mathbb R^m)\) and \(\mathcal J'(\boldsymbol u)\in W^{-1,p'}(\Omega;\mathbb R^m)\), we have
\[
\langle \mathcal J'(\boldsymbol u),\boldsymbol u_n-\boldsymbol u\rangle=o(1).
\]
Thus,
\begin{equation}\label{PS condition eq1}
  \begin{aligned}
       o(1)&=\langle\mathcal J'(\boldsymbol{u}_n)-\mathcal J'(\boldsymbol{u}),\boldsymbol{u}_n-\boldsymbol{u}\rangle=\int_\Omega |D\boldsymbol{u}_n|^{p-2}D\boldsymbol{u}_n\cdot D(\boldsymbol{u}_n-\boldsymbol{u})\\
    &\quad-\int_\Omega |D\boldsymbol{u}|^{p-2} D\boldsymbol{u}\cdot D(\boldsymbol{u}_n-\boldsymbol{u})+o(1)\\
    &=\int_\Omega\left(|D\boldsymbol{u}_n|^{p-2} D\boldsymbol{u}_n-|D\boldsymbol{u}|^{p-2} D\boldsymbol{u}\right)\cdot D(\boldsymbol{u}_n-\boldsymbol{u})+o(1).
  \end{aligned}  
\end{equation}
We recall the following vectorial inequalities for the vector field $\mathcal X:\R^k\to\R^k$, $\mathcal X(x)=|x|^{p-2}x$.
    \begin{align}\label{vectorial inequality1}
      \left(\mathcal X(x_1)-\mathcal X(x_2)\right)\cdot(x_1-x_2)&\ge 2^{2-p}|x_1-x_2|^p,\ p\ge2,\\
      \label{vector inequality2}
     \left(\mathcal X(x_1)-\mathcal X(x_2)\right)\cdot(x_1-x_2)&\ge (p-1)|x_1-x_2|^2\left(1+|x_1|^2+|x_2|^2\right)^{\frac{p-2}{2}},\ 1<p\le2,
    \end{align}
    see for instance \cite[Section 12, inequalities (I), (VII)]{lindqvist2017notes}.

We consider two cases and we use \eqref{vectorial inequality1}-\eqref{vector inequality2}:\\
\textbf{Case 1}:  $p\ge2$. From \eqref{PS condition eq1} we obtain:
\begin{equation}\label{PS condition eq2}
 o(1)\ge 2^{2-p}\int_\Omega |D\boldsymbol{u_n}-D\boldsymbol u|^p,   
\end{equation}
and $\boldsymbol{u}_n\to\boldsymbol{u}$ in $W_0^{1,p}(\Omega; \R^m)$.\\
\textbf{Case 2}:  $1<p<2$. By \eqref{PS condition eq1} we get
\begin{equation*}
    o(1)\ge(p-1)\int_\Omega|D\boldsymbol{u}_n-D\boldsymbol{u}|^2\left(1+|D\boldsymbol{u}_n|^2+|D\boldsymbol{u}|^2\right)^{\frac{p-2}{2}}.
\end{equation*}

 Now, according to  H\"older's inequality with exponents $\frac{2}{p}$ and $\frac{2}{2-p}$, we obtain:
\begin{equation}\label{PS condition eq3}
    \begin{aligned}
         \int_\Omega |D\boldsymbol{u}_n-D\boldsymbol{u}|^p&=\int_\Omega   |D\boldsymbol{u}_n-D\boldsymbol{u}|^p\cdot\left(\frac{1+|D\boldsymbol{u}_n|^2+|D\boldsymbol{u}|^2}{1+|D\boldsymbol{u}_n|^2+|D\boldsymbol{u}|^2}\right)^{\frac{p(p-2)}{4}} \\
   &\le\left(\int_\Omega |D\boldsymbol{u}_n-D\boldsymbol{u}|^2\left(1+|D\boldsymbol{u}_n|^2+|D\boldsymbol{u}|^2\right)^{\frac{p-2}{2}}\right)^{\frac{p}{2}}\times\\
   &\quad\times\left(\int_\Omega\left(1+|D\boldsymbol{u}_n|^2+|D\boldsymbol{u}|^2\right)^{\frac{p}{2}}\right)^{\frac{2-p}{2}}=o(1),
    \end{aligned}
\end{equation}
 which implies again that $\boldsymbol{u}_n\to\boldsymbol{u}$ in $W_0^{1,p}(\Omega; \R^m)$.
 
 Finally, we can conclude the proof from Theorem \ref{MP}. 

\end{proof}
Now, consider the $p$-linear perturbation: \[\boldsymbol{g}_\lambda(x,\boldsymbol{s})=\lambda|\boldsymbol{s}|^{p-2}\boldsymbol{s}+\boldsymbol{f}(x,\boldsymbol{s}),\ \ G_\lambda(x,\boldsymbol{s})=\frac{\lambda}{p}|\boldsymbol{s}|^p+F(x,\boldsymbol{s}),\]
with $\lambda\in\R$ and the respective energy functional:
\[\mathcal J_\lambda(\boldsymbol{u})=\frac{1}{p}\int_\Omega|D\boldsymbol{u}|^p-\int_\Omega G_\lambda(x,\boldsymbol{u}).\]
If $\lambda\le0$, the existence of multiple critical points is a consequence of Theorem \ref{thm 2}. 
Indeed, the quantity
\[
\|\boldsymbol u\|_\lambda:=\left(\int_\Omega|D\boldsymbol u|^p-\lambda\int_\Omega|\boldsymbol u|^p\right)^{\frac1p}
\]
is an equivalent norm for $W_0^{1,p}(\Omega; \R^m)$ and one can follow the proof of Theorem \ref{thm 2}.
Otherwise, we exploit the decomposition \eqref{decomposition}.

\begin{proof}[Proof of Theorem \ref{thm lambda}]
   {We apply Theorem \ref{MP}, due to Ambrosetti and Rabinowitz, by verifying its three assumptions.}
  
 {
   The argument is presented for $p<N$. The remaining cases are handled as follows.

   Case $p=N$: as in the proof of Theorem \ref{thm 2}, choose a sufficiently large finite exponent $\kappa>\max\{q,\mu\}$ and use the embedding \(W_0^{1,N}(\Omega;\R^m)\hookrightarrow L^\kappa(\Omega;\R^m)\), replacing $p^*$ with $\kappa$.

   Case $p>N$: as in the proof of Theorem \ref{thm 2}, use the embedding \(W_0^{1,p}(\Omega;\R^m)\hookrightarrow L^\infty(\Omega;\R^m)\).
   }
   We consider the case $\lambda>0$. We divide the proof into three steps:\\
    \textbf{Step 1}: we prove that there exist $\alpha,\rho>0$ and a finite-codimensional subspace $W\subset W_0^{1,p}(\Omega; \R^m)$ such that $\mathcal J_\lambda\ge\alpha$ on $\partial B_\rho\cap W$.
    
    {Let $W_{h_0}$ be the subspace defined in \eqref{spaces Wk}, with $h_0$ such that $\lambda<\lambda_{h_0}$.} For every $\boldsymbol{u}\in W_{h_0}\cap\partial B_\rho$, as in the proof of Theorem \ref{thm 2}-Step 1, we have:
    \begin{align*}
        \mathcal J_\lambda(\boldsymbol{u})&\ge\frac1p\int_\Omega |D\boldsymbol{u}|^p-\frac{\lambda}{p}\int_\Omega |\boldsymbol{u}|^p-C_1\|\boldsymbol{u}\|-C_2\|\boldsymbol{u}\|^q\\
        &\ge\frac{\lambda_{h_0}-\lambda}{p\lambda_{h_0}}\|\boldsymbol{u}\|^p-C_2\left(\int_\Omega|\boldsymbol{u}|^p\right)^{\frac{t}{p}}\cdot\left(\int_\Omega|\boldsymbol{u}|^{p^*}\right)^{\frac{q-t}{p^*}}-C_1\|\boldsymbol{u}\|\\
        &\ge\frac{\lambda_{h_0}-\lambda}{p\lambda_{h_0}}\|\boldsymbol{u}\|^p-C_2\lambda_{h_0}^{-\frac{t}{p}}\|\boldsymbol{u}\|^{q}-C_1\|\boldsymbol{u}\|\\
        &=\|\boldsymbol{u}\|\left[\left(\frac{\lambda_{h_0}-\lambda}{p\lambda_{h_0}}-C_2\lambda_{h_0}^{-\frac{t}{p}}\|\boldsymbol{u}\|^{q-p}\right)\|\boldsymbol{u}\|^{p-1}-C_1\right]\ge\rho
    \end{align*}
  with $\rho:=[2p\left(1+C_1\right)]^{\frac{1}{p-1}}$,  $\frac{t}{p}+\frac{q-t}{p^*}=1$ and $h_0$ large enough such that
  \[
  \frac{\lambda_{h_0}-\lambda}{p\lambda_{h_0}}-C_2\lambda_{h_0}^{-\frac{t}{p}}\rho^{q-p}\ge\frac{1}{2p}.
  \]
  
 Thus,  $\mathcal J_\lambda(\boldsymbol{u})\ge\rho$ in $\partial B_\rho\cap W_{h_0}$.\\
    \textbf{Step 2}: we prove, as in Step 2 of the proof of Theorem \ref{thm 2}, that for every finite-dimensional subspace $V\subset W_0^{1,p}(\Omega; \R^m)$ there exists $R>\rho$ such that $\mathcal J_\lambda\le0$ on $B_R^c\cap V$.\\
    \textbf{Step 3}: we prove that $\mathcal J_\lambda$ satisfies the $(PS)_c$-condition.
    
     Let $\{\boldsymbol{u}_n\}$ be a $(PS)_c$-sequence, then $\{\mathcal J_\lambda(\boldsymbol{u}_n)\}$ is bounded in $\R$ and $\langle\mathcal J_\lambda'(\boldsymbol{u}_n),\boldsymbol{\varphi}\rangle=o(1)$ for every $\boldsymbol{\varphi}\in W_0^{1,p}(\Omega; \R^m)$. Thus, there exists $C>0$ such that:
    \begin{align*}
                C+1+\|\boldsymbol{u}_n\|&\ge\mathcal J_\lambda(\boldsymbol{u}_n)-\frac{1}{\mu}\langle\mathcal J_\lambda'(\boldsymbol{u}_n),\boldsymbol{u}_n\rangle\\
        &=\left(\frac1p-\frac{1}{\mu}\right)\int_\Omega \left(|D\boldsymbol{u}_n|^p-\lambda|\boldsymbol{u}_n|^p\right)-\int_\Omega\left[F(x,\boldsymbol{u}_n)-\frac{1}{\mu}\boldsymbol{f}(x,\boldsymbol{u}_n)\cdot\boldsymbol{u}_n\right].
    \end{align*}
    Furthermore, from \eqref{g.2}, there exists a constant $C_R>0$ such that
    \[
    C+1+\|\boldsymbol u_n\|\ge\left(\frac1p-\frac{1}{\mu}\right)\int_\Omega \left(|D\boldsymbol{u}_n|^p-\lambda|\boldsymbol{u}_n|^p\right)-C_R.
    \]
Now, assume by contradiction that $\|\boldsymbol u_n\|\to+\infty$. Since
\[
C+1+\|\boldsymbol u_n\|+C_R\le \frac12\left(\frac1p-\frac1\mu\right)\|\boldsymbol u_n\|^p,
\]
as $n\to+\infty$,
 we can deduce that
\begin{equation}\label{stima dal basso norma p}
\frac12\left(\frac1p-\frac1\mu\right)\|\boldsymbol u_n\|^p\le\lambda\left(\frac1p-\frac1\mu\right)\int_\Omega|\boldsymbol u_n|^p.
\end{equation}
We define $\boldsymbol w_n:=\frac{\boldsymbol u_n}{\|\boldsymbol u_n\|}$. Since $\{\boldsymbol w_n\}$ is bounded in $W_0^{1,p}(\Omega; \R^m)$ and $\mu<p^*$, we can assume (up to a subsequence) that $\boldsymbol w_n\to \boldsymbol w$ in $L^p(\Omega; \R^m)$ and in $L^\mu(\Omega; \R^m)$. Moreover, by \eqref{stima dal basso norma p}, we obtain
\[
\int_\Omega|\boldsymbol w|^p=\lim_{n\to+\infty}\int_\Omega|\boldsymbol w_n|^p\ge\frac{1}{2\lambda}>0,\qquad\text{and $\boldsymbol w\not\equiv0$.}
\]
According to \eqref{g.2}, there exist $a_0,b_0\in L^1(\Omega)$, with $b_0(x)>0$, such that
\[
\int_\Omega F(x,\boldsymbol u_n)\ge \int_\Omega b_0(x)|\boldsymbol u_n|^\mu-\int_\Omega a_0(x)=\|\boldsymbol u_n\|^\mu\int_\Omega b_0(x)|\boldsymbol w_n|^\mu-\int_\Omega a_0(x).
\]
Since $p<\mu<p^*$, we can conclude that $\|\boldsymbol u_n\|^{-p}\displaystyle\int_\Omega F(x,\boldsymbol u_n)\to+\infty$ as $n\to+\infty$.
Thus,
\[
\frac{\mathcal J_\lambda(\boldsymbol u_n)}{\|\boldsymbol u_n\|^p}=\frac1p-\frac{\lambda}{p}\int_\Omega|\boldsymbol w_n|^p-\int_\Omega\frac{F(x,\boldsymbol u_n)}{\|\boldsymbol u_n\|^p}\to-\infty \qquad \text{as $n\to+\infty$,}
\]
and this is absurd. Hence, $\{\boldsymbol u_n\}$ is bounded in $W_0^{1,p}(\Omega; \R^m)$.

Up to a subsequence, $\boldsymbol u_n\to\boldsymbol u$ in $L^q(\Omega; \R^m)$ for every $q\in[1,p^*)$, and  \eqref{PS condition eq1} holds also in this case. 
    
    Reasoning as in the proof of Theorem \ref{thm 2}-Step 3, we obtain \eqref{PS condition eq2}-\eqref{PS condition eq3} and the strong convergence in $W_0^{1,p}(\Omega; \R^m)$ follows. 
      Finally, Theorem \ref{MP} concludes the proof.
\end{proof}

\section{Classification result}
In this section, we consider system \eqref{lane emeden system}:
\[
-\boldsymbol {\Delta}_p\boldsymbol u=\lambda|\boldsymbol u|^{p-2}\boldsymbol u+|\boldsymbol u|^{q-2}\boldsymbol u,\qquad \boldsymbol u\in W_0^{1,p}(\Omega; \R^m),
\]
 where $p,q>1$ and
$q<p^*$, and we investigate the existence of least energy solutions to
\eqref{lane emeden system}.

Let $Q_{p,q}:W_0^{1,p}(\Omega;\mathbb{R}^m)\setminus\{\boldsymbol 0\}\to\mathbb{R}$ be the Rayleigh quotient
defined by
\[
Q_{p,q}(\boldsymbol{u})
:=\frac{\displaystyle\int_\Omega |D\boldsymbol{u}|^p
-\lambda\int_\Omega|\boldsymbol{u}|^p}
{\displaystyle\left(\int_\Omega|\boldsymbol{u}|^q\right)^{\frac{p}{q}}},
\]
and define
\[
l_{p,q}
:=\inf_{\boldsymbol{u}\in W_0^{1,p}(\Omega;\mathbb{R}^m)\setminus\{\boldsymbol{0}\}}
Q_{p,q}(\boldsymbol{u})
=\inf_{\substack{\boldsymbol{u}\in W_0^{1,p}(\Omega;\mathbb{R}^m)\\
|\boldsymbol{u}|_q=1}}
\left(\int_\Omega |D\boldsymbol{u}|^p
-\lambda\int_\Omega|\boldsymbol{u}|^p\right).
\]

Let $\mathcal{J}_{p,q}:W_0^{1,p}(\Omega;\mathbb{R}^m)\to\mathbb{R}$ be the energy
functional associated with \eqref{lane emeden system}, namely
\[
\mathcal{J}_{p,q}(\boldsymbol{u})
=\frac{1}{p}\int_\Omega |D\boldsymbol{u}|^p
-\frac{\lambda}{p}\int_\Omega|\boldsymbol{u}|^p
-\frac{1}{q}\int_\Omega|\boldsymbol{u}|^q.
\]
Finally, we denote by
\[
c_{p,q}:=\inf_{X_{p,q}}\mathcal{J}_{p,q}
\]
the least energy level, where $X_{p,q}$ is the set of nontrivial weak solutions
of \eqref{lane emeden system}.

 \begin{proposition}\label{minimizer}
    Let $q>1$, $q\ne p$, and assume that $\lambda<\lambda_1$. Then $\boldsymbol{v}$, with $|\boldsymbol v|_q=1$, is a minimizer for $l_{p,q}$ if and only if $l_{p,q}^{\frac{1}{q-p}}\boldsymbol{v}$ is a least energy solution and \[l_{p,q}=\left(\frac{pq}{q-p}c_{p,q}\right)^{\frac{q-p}{q}}.\]
\end{proposition}
\begin{proof} We follow  \cite[Proposition 1.2]{saldana2022least}.
For every $\boldsymbol{u}\in X_{p,q}$ weak solution of \eqref{lane emeden system}, we define $\boldsymbol{v}:=\frac{\boldsymbol{u}}{|\boldsymbol{u}|_q}$. Hence,
\[
\int_\Omega|D\boldsymbol u|^p-\lambda\int_\Omega|\boldsymbol u|^p=\int_\Omega|\boldsymbol u|^q,
\]
and
    \begin{align*}
        l_{p,q}&\le \int_\Omega |D\boldsymbol{v}|^p-\lambda\int_\Omega|\boldsymbol{v}|^p=\frac{\displaystyle\int_\Omega |D\boldsymbol{u}|^p-\lambda\int_\Omega|\boldsymbol{u}|^p}{\displaystyle\left(\int_\Omega |\boldsymbol{u}|^q\right)^{\frac{p}{q}}}\\
        &=\left(\int_\Omega|D\boldsymbol{u}|^p-\lambda\int_\Omega|\boldsymbol{u}|^p\right)^{\frac{q-p}{q}}=\left(\frac{pq}{q-p}\mathcal J_{p,q}(\boldsymbol{u})\right)^{\frac{q-p}{q}}.
    \end{align*}
    Taking the infimum over $X_{p,q}$, we obtain:
    \begin{equation}\label{stima l_{p,q}1}
       l_{p,q}\le \left(\frac{pq}{q-p}c_{p,q}\right)^{\frac{q-p}{q}}. 
    \end{equation}
   Vice versa, if $\boldsymbol{v}\in W_0^{1,p}(\Omega; \R^m)$ is a minimizer of $l_{p,q}$ with $|\boldsymbol{v}|_q=1$, then $\boldsymbol{v}$ is a weak solution of 
    \[-\boldsymbol{\Delta}_p\boldsymbol{v}-\lambda|\boldsymbol{v}|^{p-2}\boldsymbol{v}=l_{p,q}|\boldsymbol{v}|^{q-2}\boldsymbol{v}, \qquad \boldsymbol v\in W_0^{1,p}(\Omega; \R^m).\]
     Indeed, by the Lagrange multiplier theorem, there exists $\Lambda\in\R$ such that
     \[
     \int_\Omega|D\boldsymbol v|^{p-2}D\boldsymbol v\cdot D\boldsymbol\varphi-\lambda\int_\Omega|\boldsymbol v|^{p-2}\boldsymbol v\cdot\boldsymbol \varphi=\Lambda\int_\Omega|\boldsymbol v|^{q-2}\boldsymbol v\cdot\boldsymbol \varphi,\qquad \forall\ \boldsymbol\varphi\in W_0^{1,p}(\Omega; \R^m).
     \]

    Now, testing with $\boldsymbol{\varphi}=\boldsymbol{v}$, we obtain
    \[l_{p,q}=\int_\Omega|D\boldsymbol v|^p-\lambda\int_\Omega|\boldsymbol v|^p=\Lambda,\]
     and the claim follows. 
     
    Now, the  function $\boldsymbol{u}:=l_{p,q}^{\frac{1}{q-p}}\boldsymbol{v}$ solves $-\boldsymbol{\Delta}_p\boldsymbol{u}-\lambda|\boldsymbol{u}|^{p-2}\boldsymbol{u}=|\boldsymbol{u}|^{q-2}\boldsymbol{u}$ and
    \begin{align*}
        c_{p,q}\le\mathcal J_{p,q}\left(\boldsymbol{u}\right)&=\frac{q-p}{pq}\left[\int_\Omega\left|D\boldsymbol{u}\right|^p-\lambda\int_\Omega\left|\boldsymbol{u}\right|^p\right]\\
        &=\frac{q-p}{pq}\left[\int_\Omega|D\boldsymbol v|^p-\lambda\int_\Omega|\boldsymbol v|^p \right]l_{p,q}^{\frac{p}{q-p}}\\
        &=\frac{q-p}{pq}l_{p,q}^{\frac{q}{q-p}}.
    \end{align*}
    Thus, from \eqref{stima l_{p,q}1}:
\[c_{p,q}\le\frac{q-p}{pq}l_{p,q}^{\frac{q}{q-p}}\le c_{p,q},\]
therefore\[ l_{p,q}=\left(\frac{pq}{q-p}c_{p,q}\right)^{\frac{q-p}{q}}.\]
\end{proof}
Now, we prove that \(l_{p,q}\) is achieved and that any minimizer for \(l_{p,q}\) has the form
\(\mathbf c\,\omega\), with \(\mathbf c\in S^{m-1}\). After the rescaling
\[
\widetilde\omega:=l_{p,q}^{\frac1{q-p}}\omega,
\]
the function \(\widetilde\omega\) is a positive solution of 

\begin{equation}\label{eq scalare}
\begin{cases}
-\Delta_pw=\lambda|w|^{p-2}w+|w|^{q-2}w&\text{in $\Omega$,}\\
w=0&\text{on $\partial\Omega$}.
\end{cases}
\end{equation}

\begin{proposition}\label{soluzione omega}
Let $\boldsymbol c\in S^{m-1}$ and let $\omega\in W_0^{1,p}(\Omega; \R)$. Then 
\[
\boldsymbol{v}=\boldsymbol c\omega \text{  is  a solution of \eqref{lane emeden system}}\iff \text{ $\omega$ solves \eqref{eq scalare}.}
\]
\end{proposition}
\begin{proof}
For every $\boldsymbol c\in S^{m-1}$, we have that:
\[|\boldsymbol{v}|=|\boldsymbol c|\cdot|\omega|=|\omega|,\] \[|D\boldsymbol v|=\left(\sum_{i=1}^m |c^i|^2\cdot|\nabla\omega|^2\right)^{1/2}=|\boldsymbol c|\cdot|\nabla\omega|=|\nabla\omega|. \]
 Hence,
  \begin{align*}
     \int_\Omega |D\boldsymbol{v}|^{p-2}\nabla v^i\cdot\nabla\varphi^i&=c^i\int_\Omega |\nabla\omega|^{p-2}\nabla\omega\cdot\nabla\varphi^i,\\
     \lambda\int_\Omega|\boldsymbol{v}|^{p-2}v^i\varphi^i+\int_\Omega|\boldsymbol{v}|^{q-2}v^i\varphi^i&=\lambda c^i\int_\Omega|\omega|^{p-2}\omega\varphi^i+c^i\int_\Omega|\omega|^{q-2}\omega\varphi^i,
 \end{align*}
 for every $\boldsymbol{\varphi}=(\varphi^1,\dots,\varphi^m)\in W_0^{1,p}(\Omega; \R^m)$. If $\omega$ solves \eqref{eq scalare}, we conclude that 
 \[
   \int_\Omega |D\boldsymbol{v}|^{p-2}\nabla v^i\cdot\nabla\varphi^i=     \lambda\int_\Omega|\boldsymbol{v}|^{p-2}v^i\varphi^i+\int_\Omega|\boldsymbol{v}|^{q-2}v^i\varphi^i, \qquad i=1,\dots,m,
 \]
 i.e. $\boldsymbol v$ solves \eqref{lane emeden system}. Vice versa, assume that $\boldsymbol v$ solves \eqref{lane emeden system}. Since $\boldsymbol c\in S^{m-1}$, it is not possible that $c^i=0$ for every $i=1,\dots,m$. Thus, 
 \[
 \int_\Omega|\nabla\omega|^{p-2}\nabla\omega\cdot\nabla\varphi^j=\lambda\int_\Omega|\omega|^{p-2}\omega\varphi^j+\int_\Omega|\omega|^{q-2}\omega\varphi^j,
 \]
 where $j\in\{1,\dots,m\}$ is such that $c^j\ne0$. Thus, $\omega$ solves \eqref{eq scalare}.

 \end{proof}

\begin{proof}[Proof of Theorem \ref{thm 3}]
We divide the proof into two steps:\\
\textbf{Step 1}: the infimum $l_{p,q}$ is achieved.

If $q\in(1,p)$, reasoning as in Theorem \ref{thm 1}, we have a minimizer for $\mathcal J_{p,q}$ which is a least energy solution. Indeed, the energy functional satisfies:
\[\mathcal J_{p,q}(\boldsymbol u)\ge\left(1-\frac{\lambda}{\lambda_1}\right)\frac{\|\boldsymbol u\|^p}{p}-\frac{1}{q}\int_\Omega|\boldsymbol u|^q,\]
and it is still coercive and weakly lower semicontinuous. 

On the other hand, when $q\in(p,p^*)$ Theorem \ref{thm 2} ensures that $X_{p,q}\ne\emptyset$. Let $\{\boldsymbol{u}_n\}$ be a minimizing sequence for ${\mathcal J_{p,q}}_{|_{X_{p,q}}}$. Then $\{\boldsymbol{u}_n\}$ is a Palais-Smale sequence for $\mathcal J_{p,q}$ at level $c_{p,q}$ and Step 3 of  proof of Theorem \ref{thm lambda} implies the existence of a weak solution $\boldsymbol{u}\in W_0^{1,p}(\Omega; \R^m)$ at level $c_{p,q}$. We claim that $\boldsymbol{u}\not\equiv0$.

Indeed, since $\boldsymbol u_n\in X_{p,q}$, we have
$\langle\mathcal J_{p,q}'(\boldsymbol u_n),\boldsymbol u_n\rangle=0$ for every $n$. Hence, there exist $C_1,C_2>0$, independent of $n$, such that
\[
C_1\|\boldsymbol u_n\|^p
\le\int_\Omega |D\boldsymbol u_n|^p
-\lambda\int_\Omega|\boldsymbol u_n|^p
=\int_\Omega|\boldsymbol u_n|^q
\le C_2\|\boldsymbol u_n\|^q.
\]
Since $\boldsymbol u_n\ne\boldsymbol0$, it follows that
\(
\|\boldsymbol u\|\ge
\left(\frac{C_1}{C_2}\right)^{\frac1{q-p}}
=:C_3>0.
\)
 Thus, $\boldsymbol u$ is a nontrivial least energy solution. In particular, Proposition \ref{minimizer} implies that $l_{p,q}$ is achieved.

\textbf{Step 2:} Classification of least energy solutions.

 Let $\boldsymbol{u}\in W_0^{1,p}(\Omega; \R^m)$ be a minimizer for $l_{p,q}$. We define 
\[\omega:=\sqrt{|u^1|^2+\dots+|u^m|^2}=|\boldsymbol{u}|\in W_0^{1,p}(\Omega; \R),\]
and $\boldsymbol{v}:=\widetilde{\boldsymbol{c}}\omega$ for some $\widetilde{\boldsymbol{c}}\in S^{m-1}$. 

We observe that $\nabla \omega=D\boldsymbol{u}\cdot\frac{\boldsymbol{u}}{|\boldsymbol{u}|}$ on $\{\omega\ne0\}=\{|\boldsymbol u|\ne0\}$ and
\[|\nabla \omega|\le|D\boldsymbol{u}|,\ \ \text{a.e. in $\Omega$}. \] 

     Taking $t:=|\boldsymbol v|_q^{-1}$, we obtain
    \begin{align*}
        l_{p,q}\le Q_{p,q}(t\boldsymbol{v})&=\int_\Omega|D(t\boldsymbol{v})|^p-\lambda\int_\Omega|t\boldsymbol{v}|^p=t^p\left(\int_\Omega|D\boldsymbol{v}|^p-\lambda\int_\Omega|\boldsymbol{v}|^p\right)\\
        &=\frac{\displaystyle\int_\Omega|D\boldsymbol{v}|^p-\lambda\int_\Omega|\boldsymbol{v}|^p}{\displaystyle\left(\int_\Omega|\boldsymbol{v}|^q\right)^{\frac{p}{q}}}=\frac{\displaystyle\int_\Omega|\nabla\omega|^p-\lambda\int_\Omega|\omega|^p}{\displaystyle\left(\int_\Omega|\omega|^q\right)^{\frac{p}{q}}}=\frac{\displaystyle\int_\Omega|\nabla\omega|^p-\lambda\int_\Omega|\boldsymbol u|^p}{\displaystyle\left(\int_\Omega|\boldsymbol{u}|^q\right)^{\frac{p}{q}}}\\
      &\le \frac{\displaystyle\int_\Omega |D\boldsymbol{u}|^p-\lambda\int_\Omega|\boldsymbol{u}|^p}{\displaystyle\left(\int_\Omega|\boldsymbol{u}|^q\right)^{\frac pq}} = Q_{p,q}(\boldsymbol u)=l_{p,q}.
    \end{align*}
    Therefore, $|D\boldsymbol u|=|\nabla\omega|$, $Q_{p,q}(t\boldsymbol v)=Q_{p,q}(\boldsymbol u)$ and $t\boldsymbol v$ is  a minimizer for $l_{p,q}$.
     Since
    \[
\frac{\displaystyle\int_\Omega|\nabla\omega|^p-\lambda\int_\Omega|\omega|^p}{\displaystyle\left(\int_\Omega|\omega|^q\right)^{\frac{p}{q}}}\le l_{p,q}\le \inf_{\substack{u\in W_0^{1,p}(\Omega; \R)\\ |u|_q=1}}\frac{\displaystyle\int_\Omega|\nabla u|^p-\lambda\int_\Omega|u|^p}{\displaystyle\left(\int_\Omega|u|^q\right)^{\frac{p}{q}}},
    \]
    we have  that $\omega=|\boldsymbol u|$ solves
\[
-\Delta_p\omega=\lambda\omega^{p-1}+l_{p,q}\omega^{q-1},\qquad \omega\in W_0^{1,p}(\Omega; \R),
\]
 and   $\omega>0$ by Harnack's inequality \cite{trudinger1967harnack}.  
    
  By Lagrange's identity \cite[eq. (2.1)]{hynd2023uniqueness}:
   \begin{equation}\label{lagrange's identity}
      \left|\sum_{i=1}^m\tau^iz^i\right|^2=\sum_{i=1}^m|\tau^i|^2\sum_{i=1}^m|z^i|^2-\sum_{\substack{i,j=1\\ i<j}}^m|\tau^iz^j-\tau^jz^i|^2, \quad \tau^1,\dots,\tau^m\in\R, z^1,\dots,z^m\in\R^N,
   \end{equation}
 
   we can conclude that $\boldsymbol{u}=\boldsymbol{c}\omega$ with $\boldsymbol c\in S^{m-1}$.
   
    Indeed, \eqref{lagrange's identity} with  $\tau^i=u^i$ and $z_i=\nabla u^i$,  implies that
   \[\omega^2|\nabla\omega|^2 =\omega^2|D\boldsymbol{u}|^2-\sum_{\substack{i,j=1\\ i<j}}^m| u^i\nabla  u^j- u^j\nabla  u^i|^2\quad \text{a.e. $x\in\Omega$}.\]
   Since $|D\boldsymbol{u}|=|\nabla \omega|$, we obtain  $ u^i\nabla  u^j= u^j\nabla  u^i$ for a.e. $x\in\Omega$ and for every $i,j=1,\dots,m$. Hence,
\[\omega^2\nabla u^i=\sum_{j=1}^mu^ju^j\nabla u^i=\sum_{j=1}^mu^ju^i\nabla u^j=u^i\sum_{j=1}^mu^j\nabla u^j=u^i(\omega\nabla\omega).\]
Thus, $\omega\nabla u^i=u^i\nabla\omega$ and,  
\[\nabla\left(\frac{u^i}{\omega}\right)=\frac{\nabla u^i\omega-u^i\nabla\omega}{\omega^2}=0.\]
Then the function $\frac{u^i}{\omega}=\frac{u^i}{|\boldsymbol u|}$ must be constant for every $i=1,\dots,m$. Finally, setting
\[
\widetilde\omega:=l_{p,q}^{\frac{1}{q-p}}\omega,
\]
we have that $\widetilde\omega$ solves
\[
-\Delta_p\widetilde\omega=\lambda\widetilde\omega^{p-1}+\widetilde\omega^{q-1},
\qquad \widetilde\omega\in W_0^{1,p}(\Omega; \R),
\]
and the proof is complete.
\end{proof}


\begin{thebibliography}{10}
\raggedright

\bibitem{ambrosetti1973dual}
Antonio Ambrosetti and Paul~H. Rabinowitz.
\newblock Dual variational methods in critical point theory and applications.
\newblock {\em Journal of Functional Analysis}, 14(4):349--381, 1973.

\bibitem{balci2022pointwise}
Anna~Kh. Balci, Andrea Cianchi, Lars Diening, and Vladimir Maz’Ya.
\newblock A pointwise differential inequality and second-order regularity for
  nonlinear elliptic systems.
\newblock {\em Mathematische Annalen}, 383(3):1--50, 2022.

{
\bibitem{barrenechea2023powerlaw}
Gabriel~R. Barrenechea and Endre S{\"u}li.
\newblock Analysis of a stabilised finite element method for power-law fluids.
\newblock {\em Constructive Approximation}, 57:295--325, 2023.
}


  \bibitem{candela2009infinitely}
Anna~Maria Candela and Giuliana Palmieri.
\newblock Infinitely many solutions of some nonlinear variational equations.
\newblock {\em Calculus of Variations and Partial Differential Equations},
  34(4):495--530, 2009.

\bibitem{carmona2013regularity}
Jos{\'e} Carmona, Silvia Cingolani, Pedro~J. Mart{\'\i}nez-Aparicio, and
  Giuseppina Vannella.
\newblock Regularity and Morse index of the solutions to critical quasilinear
  elliptic systems.
\newblock {\em Communications in Partial Differential Equations},
  38(10):1675--1711, 2013.

\bibitem{benedetto1989boundary}
Ya-Zhe Chen and Emmanuelle Di~Benedetto.
\newblock Boundary estimates for solutions of nonlinear degenerate parabolic
  systems.
\newblock {\em Journal für die reine und angewandte Mathematik}, 1989.

\bibitem{CianchiArma2014}
Andrea Cianchi and Vladimir~G. Maz'ya.
\newblock Global boundedness of the gradient for a class of nonlinear elliptic
  systems.
\newblock {\em Arch. Ration. Mech. Anal.}, 212(1):129--177, 2014.

\bibitem{cianchi2019optimal}
Andrea Cianchi and Vladimir~G. Maz'ya.
\newblock Optimal second-order regularity for the \(p\)-Laplace system.
\newblock {\em Journal de Math{\'e}matiques Pures et Appliqu{\'e}es},
  132:41--78, 2019.

\bibitem{correia2016semitrivial}
Sim{\~a}o Correia, Filipe Oliveira, and Hugo Tavares.
\newblock Semitrivial vs. fully nontrivial ground states in cooperative cubic
  Schr{\"o}dinger systems with $d\ge3$ equations.
\newblock {\em Journal of Functional Analysis}, 271(8):2247--2273, 2016.

{
\bibitem{costa1995perturbations}
David~G. Costa and Celius~A. Magalh\~{a}es.
\newblock Existence results for perturbations of the $p$-Laplacian.
\newblock {\em Nonlinear Analysis: Theory, Methods \& Applications},
  24(3):409--418, 1995.
}

{
\bibitem{dinca2001variational}
Gheorghe Dinc\u{a}, Petru Jebelean, and Jean Mawhin.
\newblock Variational and topological methods for Dirichlet problems with
  $p$-Laplacian.
\newblock {\em Portugaliae Mathematica}, 58(3):339--378, 2001.
}

{
\bibitem{grumiau2008asymptotics}
Christopher Grumiau and Enea Parini.
\newblock On the asymptotics of solutions of the Lane--Emden problem for the
  $p$-Laplacian.
\newblock {\em Archiv der Mathematik}, 91(4):354--365, 2008.
}

\bibitem{guedda1989quasilinear}
Mohammed Guedda and Laurent V{\'e}ron.
\newblock Quasilinear elliptic equations involving critical Sobolev exponents.
\newblock {\em Nonlinear Anal. Theory Methods Applic.}, 13(8):879--902, 1989.

\bibitem{hynd2023uniqueness}
Ryan Hynd, Bernd Kawohl, and Peter Lindqvist.
\newblock On the uniqueness of eigenfunctions for the vectorial \(p\)-Laplacian.
\newblock {\em Archiv der Mathematik}, 121(5):745--755, 2023.

{
\bibitem{irgens2014rheology}
Fridtjov Irgens.
\newblock {\em Rheology and Non-Newtonian Fluids}.
\newblock Springer, Cham, 2014.
}

\bibitem{KuusiMingionePotential}
Tuomo Kuusi and Giuseppe Mingione.
\newblock Vectorial nonlinear potential theory.
\newblock {\em J. Eur. Math. Soc. (JEMS)}, 20(4):929--1004, 2018.

\bibitem{lindqvist2017notes}
Peter Lindqvist.
\newblock {\em Notes on the \(p\)-Laplace equation}.
\newblock Number 161. University of Jyv{\"a}skyl{\"a}, 2017.


\bibitem{montoro2025regularity}
Luigi Montoro, Luigi Muglia, Berardino Sciunzi, and Domenico Vuono.
\newblock Regularity and symmetry results for the vectorial \(p\)-Laplacian.
\newblock {\em Nonlinear Analysis}, 251:113700, 2025.

\bibitem{pardo2024priori}
Rosa Pardo.
\newblock $L^\infty$ a-priori estimates for subcritical \(p\)-Laplacian equations
  with a Carath{\'e}odory non-linearity.
\newblock {\em Revista de la Real Academia de Ciencias Exactas, F{\'\i}sicas y
  Naturales. Serie A. Matem{\'a}ticas}, 118(2):66, 2024.

\bibitem{pucci2008regularity}
Patrizia Pucci and Raffaella Servadei.
\newblock Regularity of weak solutions of homogeneous or inhomogeneous
  quasilinear elliptic equations.
\newblock {\em Indiana University Mathematics Journal}, pages 3329--3363, 2008.

\bibitem{rabinowitz1}
Paul~H. Rabinowitz.
\newblock {\em Minimax methods in critical point theory with applications to
  differential equations}, volume~65 of {\em CBMS Regional Conference Series in
  Mathematics}.
\newblock Conference Board of the Mathematical Sciences, Washington, DC; by the
  American Mathematical Society, Providence, RI, 1986.

{
\bibitem{ruzicka2000electrorheological}
Michael R{\r u}\v{z}i\v{c}ka.
\newblock {\em Electrorheological Fluids: Modeling and Mathematical Theory}.
\newblock Lecture Notes in Mathematics, volume 1748. Springer-Verlag, Berlin,
  2000.
}

\bibitem{saldana2022least}
Alberto Salda\~{n}a and Hugo Tavares.
\newblock On the least-energy solutions of the pure Neumann Lane--Emden
  equation.
\newblock {\em Nonlinear Differential Equations and Applications NoDEA},
  29(3):30, 2022.

\bibitem{SchmidtMinimizer14}
Thomas Schmidt.
\newblock Partial regularity for degenerate variational problems and image
  restoration models in {BV}.
\newblock {\em Indiana Univ. Math. J.}, 63(1):213--279, 2014.

\bibitem{sciunzi2025global}
Berardino Sciunzi, Giuseppe Spadaro, and Domenico Vuono.
\newblock Global second order optimal regularity for the vectorial $ p
  $-Laplacian.
\newblock {\em arXiv preprint arXiv:2502.17067}, 2025.

\bibitem{stampacchia1965probleme}
Guido Stampacchia.
\newblock Le probl{\`e}me de Dirichlet pour les {\'e}quations elliptiques du
  second ordre {\`a} coefficients discontinus.
\newblock In {\em Annales de l'institut Fourier}, volume~15, pages 189--257,
  1965.

\bibitem{trudinger1967harnack}
Neil~S. Trudinger.
\newblock On Harnack-type inequalities and their application to quasilinear
  elliptic equations.
\newblock {\em Communications on Pure and Applied Mathematics}, 20(4):721--747,
  1967.

\bibitem{Vannella23}
Giuseppina Vannella.
\newblock Uniform {$L^\infty$}-estimates for quasilinear elliptic systems.
\newblock {\em Mediterr. J. Math.}, 20(6): Paper No. 289, 11, 2023.

\end{thebibliography}
\end{document}